\documentclass[11pt,a4paper]{amsart}
\usepackage{amssymb}
\usepackage{graphicx}

\newcommand{\begeq}{\begin{equation}}
\newcommand{\eeq}{\end{equation}}

\theoremstyle{plain}

\newtheorem{Main}{Theorem}

\newtheorem{Thm}{Theorem}[section]
\newtheorem{Proposition}[Thm]{Proposition}
\newtheorem{Lemma}[Thm]{Lemma}

\newtheorem*{Claim}{Claim}
\newtheorem{Cor}[Thm]{Corollary}

\newtheorem{Rem}[Thm]{Remark}
\newtheorem{Definition}[Thm]{Definition}

\def\max{\operatorname{max}}
\def\min{\operatorname{min}}

\begin{document}

\title[Hausdorff dimensions and entropies]
{Exceptional sets for average conformal dynamical systems}
\author{Congcong Qu, Juan Wang\textsuperscript{*}}
\address{Congcong Qu, College of Big Data and software Engineering, Zhejiang Wanli University, Ningbo, 315107, Zhejiang, P.R.China}
\email{congcongqu@foxmail.com}
\address{Juan Wang, School of mathematics, physics and statistics, Shanghai University of Engineering
Science, Shanghai, 201620, P.R. China}
\email{wangjuanmath@sues.edu.cn}

\newcommand\blfootnote[1]{%
\begingroup
\renewcommand\thefootnote{}\footnote{#1}%
\addtocounter{footnote}{-1}%
\endgroup
}

\thanks{
The second auther is partially supported by NSFC (11871361, 11801395) and the Talent Program
of Shanghai University of Engineering Science.
}



\blfootnote{\textsuperscript{*}Corresponding author}

\begin{abstract}
Let $f: M \to M$ be a $C^{1+\alpha}$ map/diffeomorphism of a compact Riemannian manifold $M$ and $\mu$ be an expanding/hyperbolic ergodic $f$-invariant Borel probability measure on $M$. Assume $f$ is average conformal expanding/hyperbolic on the support set $W$ of $\mu$ and  $W$ is locally maximal. For any subset $A\subset W$ with small entropy or dimension,  we investigate the topological entropy and Hausdorff dimensions of the $A$-exceptional set and the limit $A$-exceptional set.
\end{abstract}

\maketitle

\tableofcontents

\section{Introduction and the main results}
\subsection{Background}

The present paper is motivated by \cite{Do}, \cite{CG1}, \cite{CG2} and \cite{WWCZ}.
Let $M$ be a compact Riemannian manifold and $f:M\rightarrow M$ be a continuous transformation. Let $W\subset M$ be $f$-invariant, that is $f(W)=W$. Given $A\subset W$, the $A$-{\it exceptional set} for $f|_{W}$ in $W$ is defined as
$$E^{+}_{f|_W}(A)=\{x\in W:\overline{\mathcal{O}_{f}(x)}\cap A={\varnothing}\},$$
where $\mathcal{O}_{f}(x)$ denotes the forward orbit of the point $x$ under the action of $f$, that is $\mathcal{O}_f(x)=\{f^n(x): n\in\mathbb{N}\cup\{0\}\}$,
and $\overline{\mathcal{O}_{f}(x)}$ is the closure of the set $\mathcal{O}_{f}(x)$.
The {\it limit} $A$-{\it exceptional set} for $f|_{W}$ in $W$ is defined as
$$I^{+}_{f|_W}(A)=\{x\in W:\omega_{f}(x)\cap A={\varnothing}\},$$
where $\omega_{f}(x)=\bigcap_{n\in \mathbb{N}}\overline{\{f^{k}(x)|k\geq n\}}$, and $\omega_{f}(x)$ is the set of limit points of $\mathcal{O}_{f}(x)$.
Let $\mu$ be an expanding (hyperbolic) ergodic $f$-invariant Borel probability measure on $M$ and $W$ is the support set of $\mu$.
Assume $f$ is average conformal expanding (average conformal hyperbolic) on $W$ and $W$ is locally maximal.
Roughly speaking, the average conformal expanding map (average conformal hyperbolic diffeomorphism) $f|_W$ is uniformly expanding  (not necessarily uniformly hyperbolic) and possesses only one positive Lyapunov exponent  (only one positive and one negative Lyapunov exponents) for any ergodic $f$-invariant measure on $W$ (See Definition \ref{defofac}).
For any subset $A\subset W$ with small entropy or dimension,  we investigate the topological entropy and Hausdorff dimensions of the $A$-exceptional set and the limit $A$-exceptional set.


Let $f: [0,1)\to[0,1)$ be the Gauss map,
\[f(x)=\begin{cases}
\frac{1}{x}-[\frac{1}{x}],&0<x<1,\\
0,&x=0,
\end{cases}\]
where $[\frac{1}{x}]$ denotes the integer part of $\frac1x$. For $x\in[0,1)$, $x$ is \emph{badly approximable} if there is a positive constant $c=c(x)$ such that for any reduced rational number $\frac{p}{q}$ we have~$|\frac{p}{q}-x| >\frac{c}{q^{2}}$. It is not difficult to see $x$ is badly approximable if and only if $x$ is in the $\{0\}$-exceptional set for the Gauss map $f$.
It is straightforward to check that the set of badly approximable numbers must have Lebesgue measure zero.
However Besicovitch and Jarnik \cite{Jarnik} proved that the Hausdorff dimension of the $\{0\}$-exceptional set is $1$. Abercrombie and Nair \cite{AN} proved the analogous result in the case of a Markov map on the interval. Let $J$ be the Julia set of an expanding rational map $T$ of degree at least $2$ on the Riemannian sphere. For each independent set $A\subset J$, (that is $A$ is finite and for pairs of not necessarily distinct points $z_1, z_2\in A$ satisfying $T^n(z_1)\neq z_2$ for each $n\in\mathbb{N}$), Abercrombie and Nair \cite{AN2} stated that the $A$-exceptional set has Hausdorff dimension equal to that of the Julia set $J$. Their approach is to construct a certain measure supported on the set of points whose forward orbit misses certain neighborhoods of $A$. Then they proved that the Hausdorff dimension of the $A$-exceptional set was arbitrarily close to $\text{dim}_HJ$.
Dani \cite{Dani} investigated special countable subsets $A\subset T^n$, $n\geq 1$, and showed that the $A$-exceptional set under a hyperbolic automorphism has full Hausdorff dimension $n$.
Dani's method relied on the theory of Schmidt games, which were first introduced by Schmidt in \cite{Schmidt}.
It is specific for the algebraic case but is rather inapplicable to general
Anosov diffeomorphisms.
For a transitive $C^2$ Anosov diffeomorphism $f: M\to M$ of a compact Riemannian manifold $M$, Urba$\acute{\text{n}}$ski \cite{U} showed that the Hausdorff dimension of the set of points with non-dense orbit under $f$ equals to the dimension of $M$. For an expanding endomorphism he also proved the same result. His approach is using elementary properties of Markov partitions for hyperbolic dynamical systems and a general McMullen's result (see \cite{Mcm}).
For a $C^2$-expanding self-maps on the circle, it was proved in \cite{Tseng1} that for the set of points with non-dense forward orbit is a winning
set for Schmidt games. Hence it has  full Hausdorff dimension. Later Tseng \cite{Tseng2} also proved that a certain Anosov diffeomorphism on the 2-torus has a winning
non-dense set and thus has full Hausdorff dimension. And he answered the question of whether there are non-algebraic dynamical systems with winning non-dense set in dimensions greater than one. Wu \cite{Wu1, Wu2, Wu3} extended Tseng's results to the general expanding endomorphisms and partially hyperbolic diffeomorphisms.
He used the tool of Schmidt game and the modified Schmidt game, which was introduced by Kleinbock and Weiss in \cite{KW}.

Notice that these results above are concerned with the exceptional set of a small set with at most countably many points. A breakthrough result is due to Dolgopyat.
Dolgopyat \cite{Do} considered the one-sided shift space, piecewise uniformly expanding maps of the interval,  Anosov diffeomorphisms on the surface, conformal  Anosov flows, and geodesic flows on Riemannian surfaces of negative curvature. He exploited the above methods in an essential way to get that, for a set with smaller topological entropy or Hausdorff dimension than the corresponding quantities of the map, its exceptional set has full entropy or full Hausdorff dimension respectively.

Basing on Dolgopyat's work, Campos and  Gelfert \cite{CG1,CG2} got the corresponding results for conformal $C^{1+\alpha}$ dynamical systems. 
In \cite{CG1}, Campos and Gelfert considered a rational map  $f$ of degree $d\geq2$ on the  Riemannian sphere, and $A\subset J$. They proved that if the topological entropy of $A$ is smaller than the topological entropy of the full system, then the $A$-exceptional set has full topological entropy. Moreover, if the Hausdorff dimension of $A$ is smaller than the dynamical dimension of $f$, then the Hausdorff dimension of the $A$-exceptional set is larger than or equal to the dynamical dimension.
For the case of nonuniformly hyperbolic and conformal diffeomorphisms, Campos and  Gelfert \cite{CG2} have the same conclusion. 

It is natural to ask whether Campos and Gelfert's results still hold for nonconformal dynamical systems? 
In this paper, we obtain the results for ~$C^{1+\alpha}$ average conformal expanding/hyperbolic dynamical systems. The notion of the average conformal maps was introduced by Ban, Cao and Hu in \cite{BCH} for repellers, which generalized the concept of the quasi-conformal maps introduced by Barreira  in \cite{Barreira} and the concept of weakly conformal maps in \cite{Pesin}.
Wang, Wang, Cao and Zhao introduced the concept of average conformal hyperbolic sets in \cite{WWCZ}. In \cite{ZCB}, the authors gave an example which is average conformal but not conformal. So average conformal dynamical systems are indeed more general than conformal ones.

\subsection{Statements of the main results}




Now we give the precise statements of our main results. First of all, we consider a $C^{1+\alpha}$ map $f: M \to M$ of a compact Riemannian manifold $M$, and an expanding ergodic $f$-invariant Borel probability measure $\mu$ on $M$. Assume $f$ is average conformal expanding (see the definition \ref{defofac}) on the support set $W$ of $\mu$ and $A\subset W$. Our first main result is that if the Hausdorff dimension of $A$ is smaller than the Hausdorff dimension of $\mu$, then the topological entropy and the Hausdorff dimension of the $A$-exceptional set are not smaller than the measure-theoretic entropy and the Hausdorff dimension of $\mu$ respectively. Moreover, if the topological entropy of $A$ is smaller than the measure-theoretic entropy of $\mu$, then the Hausdorff dimension of the $A$-exceptional set is not smaller than the Hausdorff dimension of $\mu$.

\begin{Main}\label{main1}
Let $f: M \to M$ be a $C^{1+\alpha}$ map of a compact Riemannian manifold $M$ and $\mu$ be an expanding ergodic $f$-invariant Borel probability measure on $M$. Denote the support set of $\mu$ by $W$. Suppose that $W$ is locally maximal and $f$ is average conformal expanding on $W$.
\begin{itemize}
  \item [(i)] For every subset $A\subset W$ such that $\dim_{H} A<\dim_{H} \mu$, then
               $$ h(f|_{W}, E^{+}_{f|_{W}}(A))\geq h_{\mu}(f|_{W})\quad \text{and}\quad \dim_H E^+_{f|_W}(A)\geq \dim_H\mu.$$
  \item [(ii)] For every subset  $A\subset W$ such that  $h(f|_{W}, A) < h_{\mu}(f|_{W})$, then
               $$\dim_{H}E^{+}_{f|_{W}}(A) \geq \dim_{H} \mu.$$
\end{itemize}
\end{Main}


Let $DD(f|_{W})$ be the dynamical dimension, which we recall the definition in Definition \ref{defofDD}.
The following Corollary shows that if the Hausdorff dimension of $A$ is smaller than the dynamical dimension of $f|_W$, then the Hausdorff dimension of the $A$-exceptional set is larger than or equal to the dynamical dimension of $f|_W$. Campos and Gelfert \cite{CG1} have the same conclusion in the case of nonuniformly expanding maps, which are conformal.

\begin{Cor}\label{dd}
Let $f: M \to M$ be a $C^{1+\alpha}$ map of a compact Riemannian manifold $M$ and $W \subset M$ be a compact $f$-invariant locally maximal subset. If $f$ is average conformal expanding on $W$, then for every subset $A\subset W$ with  $\dim_{H}A < DD(f|_{W})$, we have
$$\dim_{H}E^{+}_{f|_{W}}(A) \geq DD(f|_{W}).$$
\end{Cor}

In particular, if the Hausdorff dimension of $A$ is smaller than the Hausdorff dimension of $W$, then  the $A$-exceptional set has full Hausdorff dimension.
The following result gives a generalization  of the main Theorem  in \cite{U}.

\begin{Cor} \label{urbanski}
 Let $M$ be a compact Riemannian manifold and $ f: M \to M$ be a $C^{1+\alpha}$ map. Suppose $W\subset M$ is an average conformal repeller (see the definition \ref{defofacr}). Then for every subset $A\subset W$ with  $\dim_H A  < \dim_H W,$ we have
$$ \dim_H E^{+}_{f|_{W}}(A) = \dim_H W.$$
\end{Cor}

Our second work is to consider a $C^{1+\alpha}$ diffeomorphism $f: M \to M$ of a compact Riemannian manifold $M$. We assume $f$ is average conformal hyperbolic (see the definition \ref{defofac}) on the support set $W$ of a hyperbolic ergodic $f$-invariant Borel probability measure $\mu$ and obtain the following results. 

\begin{Main}\label{main2}
Let $f: M \to M$ be a $C^{1+\alpha}$ diffeomorphism of a $d$-dimensional compact Riemannian manifold $M$ and $\mu$ be a hyperbolic ergodic $f$-invariant Borel probability measure on $M$. Denote the support set of $\mu$ by $W$. Suppose $W$ is locally maximal and $f$ is average conformal hyperbolic on $W$.
\begin{itemize}
  \item [(i)] For every subset $A\subset W$ such that  $\dim_{H}A<\dim_{H}~\mu$, then
              $$h(f|_{W}, I^{+}_{f|_{W}}(A))\geq h_{\mu}(f)\quad \text{and}\quad \dim_H I^+_{f|_W}(A) = \dim_{H}E^{+}_{f|_{W}}(A)\geq \dim_{H}~\mu.$$
  \item [(ii)] For every subset $A\subset W$ such that  $h(f|_{W}, A)<h_{\mu}(f|_{W})$, then
             $$ \dim_{H}I^{+}_{f|_{W}}(A)\geq \dim_{H}~\mu.$$
\end{itemize}
\end{Main}

\begin{Rem}
	For technical reasons, we only state the second item of this result in terms of the dimension of the limit $A$-exceptional sets. It is still unknown for us whether this holds for $A$-exceptional sets.
\end{Rem}

The following Corollary extends Theorem C in \cite{CG2} to average conformal hyperbolic diffeomorphisms.  


\begin{Cor}\label{ddd}
Let $f: M \to M$ be a $C^{1+\alpha}$ diffeomorphism of a $d$-dimensional compact Riemannian manifold $M$ and $W \subset M$ be a compact $f$-invariant locally maximal subset. If $f$ is average conformal hyperbolic on $W$,
then for any subset $A\subset W$ such that  $\dim_{H}A<DD(f|_{W}),$  we have
$$ \dim_{H}E^{+}_{f|_{W}}(A)\geq DD(f|_{W}).$$
\end{Cor}




This paper is organized as follows. In Section $2$ we recall some definitions and some preparatory results we need to use. In Section $3$ we give the proofs of Theorem \ref{main1}, Corollary \ref{dd} and Corollary \ref{urbanski}. In Section $4$ we give the proofs of Theorem \ref{main2} and Corollary \ref{ddd}.

\section{Preliminaries}
In this section we recall some notions  and basic facts from dynamical systems and dimension theory.

\subsection{Dimensions of sets and measures} We recall some notions and basic facts from dimension theory, see the book  \cite{Pesin} for a detailed introduction.

Let $X$ be a compact metric space equipped with a metric $d$. Given  a subset $Z$ of $X$, for $s\geq 0$ and $\delta>0$, define
\[
\mathcal{H}_{\delta}^{s}(Z)=\inf \left\{\sum_{i}|U_i|^s: \
Z\subset \bigcup_{i}U_i,~|U_i|\leq \delta\right\}
\]
where $|\cdot|$ denotes the diameter of a set. The quantity
$\mathcal{H}^{s}(Z):=\lim\limits_{\delta\rightarrow 0}\mathcal{H}_{\delta}^{s}(Z)$ is called the {\em $s$-dimensional Hausdorff measure} of $Z$. Define the {\em Hausdorff dimension} of $Z$, denoted by $\dim_H  Z$, as follows:
\[
\dim_H  Z =\inf \{s:\ \mathcal{H}^{s}(Z)=0\}=\sup \{s: \mathcal{H}^{s}(Z)=\infty\}.
\]

The Hausdorff dimension has the following properties:
\begin{enumerate}
\item [(i)]
 Monotonicity: for $Y_{1}\subset Y_{2}\subset X$, we have ~$\dim_{H}Y_{1}\leq \dim_{H}Y_{2}$.
\item [(ii)]  Countable stability:~$\dim_{H}(\bigcup_{i=0}^{\infty}B_{i})=\sup_{i}\dim_{H}B_{i}$.
\end{enumerate}

The following  lemma is well-known in the field of fractal geometry,  see Falconer's book \cite{F} for proofs.
\begin{Lemma}\label{H}
Let X and Y be metric spaces, and $\Phi: X \rightarrow Y$ be an onto and $(C, r)$-H\"{o}lder continuous map for some $C > 0$ and $0<r<1$. Then $\dim_{H} Y\leq r^{-1}\dim_{H} X$.
\end{Lemma}

\begin{Lemma}\label{dimensionlemma} \cite{CG2}
Suppose  $B_{1}, B_{2}$ are metric spaces and $E \subset B_{1}\times B_{2}$. If there exist two numbers  $b_{1}, b_{2}$ such that $\dim_{H} B_{1} \geq b_{1}$, and for any  $y\in B_{1}$,  $$\dim_{H}\left(E \cap (\{y\}\times B_{2})\right)\geq b_{2},$$
then  $\dim_{H} E \geq b_{1}+b_{2}$.
\end{Lemma}

\begin{Definition}
Let $\mu$ be a probability measure on $X$. The Hausdorff dimension of $\mu$ is defined as $$\dim_{H} \mu=\inf\{\dim_{H} Y: Y \subset X , ~\mu(Y)=1\}.$$
\end{Definition}

\begin{Definition}\label{defofDD}
Given a map $f:X\rightarrow X$ and a compact invariant set $W\subset X$, the dynamical dimension of $f|_{W}$ is defined as  $DD(f|_{W})=\sup \dim_{H}\mu$, where the supremum is taken over all the ergodic measure $\mu$ supported on $W$ with $h_{\mu}(f|_{W})>0$.
\end{Definition}

\subsection{Topological entropy and topological pressure}

We next recall the definition of the entropy for a general subset, not necessarily invariant or compact. See \cite{Bowen} for more details.

\begin{Definition}\label{deentropy}
Let $X$ be a compact metric space and $f: X\rightarrow X$ be a continuous surjective map. Suppose $\mathcal{A}=\{ A_{1}, A_{2}, ..., A_{n}\}$ is a finite open cover of $X$.
Given a subset $U \subset X$, we say $U \prec \mathcal{A}$ if there exists $ j \in \{1,2,...,n\}$ such that $U \subset  A_{j} \in \mathcal{A}$, otherwise we denote it by $U\nprec \mathcal{A}$.
Define
\[n_{f, \mathcal{A}}(U)=\left\{
		  \begin{array}{lll}
		  0,  &\quad \mbox{if} ~U\nprec \mathcal{A};\\
          \infty,  &\quad \mbox{if}~ f^{k}(U)\prec \mathcal{A}, ~\forall ~k\in \mathbb{N} ;\\
		  l, &\quad \mbox{if}~f^{k}(U)\prec \mathcal{A}, ~\forall ~k\in \{0, 1, ..., l-1\}~and ~f^{l}(U)\nprec \mathcal{A}.
		  \end{array}\right.\]
Given a subset $Y\subset X$, let $\mathcal{U}=\{U_1, U_2, \cdots \}$ be an open cover of $Y$. For any $d\geq 0$, set
\begin{displaymath}
m(\mathcal{A}, d, \mathcal{U})\triangleq \sum_{U_i\in\mathcal{U}}e^{-d\cdot n_{f, \mathcal{A}}(U_i)}\\
\end{displaymath}
and
\begin{eqnarray*}
\begin{aligned}
m_{\mathcal{A}, ~d}(Y)\triangleq\lim_{\rho\rightarrow 0}\inf&\Big\{m(\mathcal{A}, d, \mathcal{U}): \mathcal{U}=\{U_1, U_2, \cdots \},\\
&Y\subset \bigcup _{U_i\in \mathcal{U}} U_i, ~e^{-n_{f, \mathcal{A}}(U_i)}<\rho
\text{ for each }U_i\in \mathcal{U}\Big\}.
\end{aligned}
\end{eqnarray*}
Then define$$h_{\mathcal{A}}(f, Y)= \inf\{d:m_{\mathcal{A}, ~d}(Y)=0\}=\sup\{d:m_{\mathcal{A}, ~d}(Y)=\infty\}.$$
The following quantity
$$h(f, Y)= \sup_{\mathcal{A}}h_{\mathcal{A}}(f, Y)$$
is called the topological entropy of $f$ on the subset $Y$.
\end{Definition}

The topological entropy has the following basic properties:
\begin{enumerate}
\item [(i)]  If $f:X\rightarrow X$ is semi-conjugate to $g:Y\rightarrow Y$ , i.e. there exists a continuous surjective $\pi:X\rightarrow Y$ such that $g\circ \pi=\pi \circ f$, then for any subset $A\subset X$ we have $h(g, \pi(A))\leq h(f, A)$.
\item [(ii)] ~For any subset $A\subset X$, we have~$h(f, f(A))=h(f, A)$.
\item [(iii)]~Countable stability: ~$h(f, \bigcup _{i=1}^{\infty}A_{i})=\sup_{i}h(f, A_{i})$.
\item [(iv)]~For any~$m\in \mathbb{N}$, any ~$A\subset X$, we have ~$h(f^{m}, A)=m~h(f, A)$.
\item [(v)]~Monotonicity: if~$A\subset B\subset X$, then~$h(f, A)\leq h(f, B)$.
\item [(vi)]~Variational principle: $h(f)=\sup_{\mu\in\mathcal{M}_{erg}(f|_X)}h_{\mu}(f)$, where~$\mathcal{M}_{erg}(f|_X)$ denotes the set of all the ergodic measures for ~$f$. Here $h(f)$ is $h(f, X)$ in shorthand.
\end{enumerate}

Let $f: X\rightarrow  X$ be a continuous transformation on a compact metric space $(X,d)$. A subset $F\subset X$ is called  an \emph{$(n,\varepsilon)$-separated set with respect to $f$} if for any $x,y\in F,x\neq y$, we have $d_{n}(x,y)\triangleq\max_{0\leq k\leq n-1} d(f^{k}(x),f^{k}(y))>\varepsilon$. And a sequence of continuous functions $\varPhi=\{\varphi_{n}\}_{n\geq 1}$ are called {\it sub-additive} if
\begin{displaymath}
\varphi_{n+m}\leq \varphi_{n}+\varphi_{m}\circ f^{n}, \forall n,m \in \mathbb {N}.
\end{displaymath}
Furthermore, a sequence of continuous functions $\varPhi=\{\varphi_{n}\}_{n\geq 1}$ is called \emph{ super-additive }
if $-\varPhi=\{-\varphi_{n}\}_{n\geq 1}$ is sub-additive.

\begin{Definition}
Let $Z$ be a subset of $X$ and $\varPhi=\{\varphi_{n}\}_{n\geq 1}$ a sequence of sub-additive/supper-additive potentials  on $X$, put
\begin{displaymath}
P_{n}(Z,f,\varPhi,\varepsilon)=\sup\{\sum_{x\in F} e^{\varphi_{n}(x)}|F\subset Z~ \text{ is an } (n,\varepsilon)\text{-separated set}\}.
\end{displaymath}
The upper sub-additive/supper-additive topological pressure of $\varPhi=\{\varphi_{n}\}_{n\geq 1}$ with respect to $f$ on the set $Z$  is defined as
\begin{equation}\label{upper pressure}
\overline{P}_{Z}(f,\varPhi)=\lim_{\varepsilon\rightarrow 0}\limsup_{n\rightarrow \infty}\frac{1}{n}\log P_{n}(Z,f,\varPhi,\varepsilon).
\end{equation}
\end{Definition}

\begin{Rem}
Consider \emph{liminf} instead of \emph{limsup} in \emph{(\ref{upper pressure})}, we get a quantity $\underline{P}_{Z}(f,\varPhi)$
 which is called \emph{lower sub-additive/super-additive topological pressure of $\varPhi=\{\varphi_{n}\}_{n\geq 1}$  with
 respect to $f$ on $Z$}. For any compact invariant set $Z$, we have \emph{$\underline{P}_Z(f, \varPhi)=\overline{P}_Z(f, \varPhi)$}. The common value is denoted by $P_Z(f, \varPhi)$, which is called the \emph{sub-additive/\\
 super-additive topological pressure of $\varPhi$ with respect to $f$ on $Z$}. See \cite{Barreira, Pesin} for proofs.
\end{Rem}

\begin{Definition}
For any given potential  $\varphi: X\rightarrow \mathbb{R}$, subset  $Z\subset X$, $\delta >0$ and  $N\in \mathbb{N}$, denote  $\mathcal{P}(Z, N, \delta)$ the collection of countably many sets  $\Big\{(x_{i}, n_{i})\subset Z\times \{N, N+1, ...\}\Big\}$ such that
$Z\subset \bigcup_{i}B_{n_{i}}(x_{i}, \delta)$,
where
$$ B_{n_{i}}(x_{i}, \delta)=\{y\in X: d(f^{j}(x_{i}), f^{j}(y))<\delta, \ j=0, 1, ..., n_{i}-1\}.$$
Given $ s\in \mathbb{R}$, define
\begin{eqnarray*}
\begin{aligned}
m_{P}(Z, s, \varphi, N, \delta)&=\inf_{\mathcal{P}(Z, N, \delta)}\sum_{(x_{i}, n_{i})}\exp(-n_{i}s+S_{n_{i}}\varphi(x_{i})),\\
m_{P}(Z, s, \varphi, \delta)&=\lim_{N\rightarrow \infty}m_{P}(Z, s, \varphi, N, \delta).
\end{aligned}
\end{eqnarray*}
Notice that $m_{P}(Z, s, \varphi, \delta)$ is non-increasing in ~$s$, and takes values $\infty$ and $0$ at all but at most one value of $s$. Denote the critical value of $s$ by
\begin{align*}
P_{Z}(\varphi, \delta)&=\inf\{s\in \mathbb{R}|m_{P}(Z, s, \varphi, \delta)=0\}\\
&=\sup\{s\in \mathbb{R}|m_{P}(Z, s, \varphi, \delta)=\infty\}.
\end{align*}
Therefore, for $s<P_{Z}(\varphi, \delta)$, we have $m_{P}(Z, s, \varphi, \delta)=\infty$. And for $s>P_{Z}(\varphi, \delta)$, we have $m_{P}(Z, s, \varphi, \delta)=0$.
The topological pressure of the potential ~$\varphi$ with respect to $f$ on a set $Z$ is defined as
$$P_{Z}(f, \varphi)=\lim_{\delta \rightarrow 0}P_{Z}(\varphi, \delta).$$
\end{Definition}

\begin{Rem}
If $\phi = 0,$ then $P_{Z}(f, 0)$ is the topological entropy of $f$ on $Z$, and we denote it by $h(f, Z)$. Notice that this definition is equivalent to that in Definition \ref{deentropy}. For readers who are not familiar with this, one can refer to \cite{PP}.
\end{Rem}

\begin{Rem}
If $\varPhi=\{\varphi_n\}_{n\geq1}$ is \emph{additive} in the sense that $\varphi_n(x)=\varphi(x)+\varphi(f(x))+\cdots+\varphi(f^{n-1}(x))$ for some continuous function $\varphi: X\to \mathbb{R}$, we simply denote the topological pressures $\underline{P}_Z(f, \varPhi)$ and $\overline{P}_Z(f, \varPhi)$ as $\underline{P}_Z(f, \varphi)$ and $\overline{P}_Z(f, \varphi)$ respectively. For any compact invariant set $Z\subset X$ we have $P_Z(f, \varphi)=\underline{P}_Z(f, \varphi)=\overline{P}_Z(f, \varphi)$. See \cite{Pesin} for proofs.
\end{Rem}

Let $\mathcal{M}(X)$ be the space of all Borel probability measures on $X$ endowed with the $weak^{\ast}$ topology. Let $\mathcal{M}_{inv}(f|_X)$ denote the subspace of $\mathcal{M}(X)$ consisting of all $f-$invariant measures.
For $\mu\in\mathcal{M}_{inv}(f|_X)$, let $h_\mu(f)$ denote the entropy of $f$ with
respect to $\mu$.
The authors in \cite{CFH} proved the following variational principle.

\begin{Thm}
Let $f:X\rightarrow X$ be a continuous transformation on a compact metric space $X$ and $\varPhi=\{\varphi_{n}\}_{n\geq 1}$ be a sequence of sub-additive potential on $X$. We have
\begin{displaymath}
P_{X}(f,\varPhi)=\sup\{h_{\mu}(f)+\varPhi_{\ast}(\mu):\mu\in\mathcal{M}_{inv}(f|_X),~\varPhi_{\ast}(\mu)\neq-\infty\},
\end{displaymath}
where $\varPhi_*(\mu)\triangleq\displaystyle\lim_{n\to\infty}\frac{1}{n}\int\varphi_n d\mu$. (The existence of the limit follows from a sub-additive argument.)
\end{Thm}

\subsection{Conformal measures and average conformal dynamical systems}

We review the ~Oseledec's Theorem which contains the definition of the Lyapunov exponents. Particularly, we give the definitions of conformal measure and average conformal dynamical systems (See Definition \ref{defofac}).

\begin{Thm} \cite{Oseledets} Suppose $f$ is  a  $C^{1}$ map on a $d$-dimensional compact Riemannian manifold $M$ and $\mu$ is an  $f$-invariant Borel probability measure, then there exists an $f$-invariant set  $R$  with $\mu(R)=1$ such that for each $x\in R$,
\begin{enumerate}
\item [(i)]  There exists a measurable filtration
\begin{displaymath}
\{0\}=E_{0}(x)\subset E_{1}(x)\subset...\subset E_{\ell(x)}(x)=T_{x}M,
\end{displaymath}
where $\{E_{i}(x)\}_{1\leq i\leq \ell(x)}$ are linear subspaces of $T_{x}M$ and $0<\ell(x)\leq d$ such that
\begin{displaymath}
 \quad D_{x}f E_{i}(x)=E_{i}(f(x)).
\end{displaymath}
\item [(ii)] There exist measurable functions $\lambda_{1}(x)<\lambda_{2}(x)<...<\lambda_{\ell(x)}(x)$, such that for any $v\in E_{i}(x)\setminus E_{i-1}(x), 1\leq i\leq \ell(x)$, we have~$$\lambda_{i}(x)=\lim_{n\rightarrow \infty}\frac{1}{n}\log\|D_{x}f^{n}(v)\|.$$
\end{enumerate}

Furthermore, if $f$ is a $C^{1}$ diffeomorphism, then  there exists an $f$-invariant set $Y$  with $\mu(Y)=1$ such that for each $x\in Y$,
\begin{enumerate}
\item [(i)] There exists an $Df$-invariant splitting of the tangent space
\begin{displaymath}
T_{x}M=E_{1}(x)\oplus E_{2}(x)\oplus...\oplus E_{\ell(x)}(x).
\end{displaymath}
\item [(ii)] There exist measurable functions $\lambda_{1}(x)<\lambda_{2}(x)<...<\lambda_{\ell(x)}(x)$, such that for any ~$v\in E_{i}(x), ~v\neq 0$, we have~$$\lambda_{i}(x)=\lim_{n\rightarrow \infty}\frac{1}{n}\log\|D_{x}f^{n}(v)\|.$$
\end{enumerate}
\end{Thm}

If $\mu$ is an ergodic measure for $f$, then $\lambda_{i}(x)$ and $\ell(x)$ are independent of the choice of the point $x$. We denote them by $\lambda_{1}(\mu), \lambda_{2}(\mu), ...,  \lambda_{\ell}(\mu)$ and we call these numbers the {\it Lyapunov exponents} of $f$ with respect to the measure $\mu$. $\mu$ is said to be \emph{expanding} if $0<\lambda_1(\mu)<\lambda_2(\mu)<\cdots<\lambda_\ell(\mu)$. We say $\mu$ is \emph{hyperbolic} if $\lambda_1(\mu)<\cdots<\lambda_k(\mu)<0<\lambda_{k+1}(\mu)<\cdots<\lambda_\ell(\mu)$ for some $k$, $1\leq k<\ell$.


\begin{Definition}\label{defofac}
\begin{enumerate}
\item Let $f: M \to M$ be a $C^1$ map of a compact Riemannian manifold $M$ and $\mu$ be an expanding $f$-invariant ergodic Borel probability measure on $M$. $\mu$ is called conformal if   all the Lyapunov exponents with respect to ~$\mu$ are equal, that is $\lambda_1(\mu)=\lambda_2(\mu)=\cdots=\lambda_\ell(\mu)>0$. Then we say $f$ is an average conformal expanding map on $M$ if every $f$-invariant ergodic measure is expanding and conformal.

\item Let $f: M \to M$ be a $C^1$ diffeomorphism of a compact Riemannian manifold $M$ and $\mu$ be a hyperbolic $f$-invariant ergodic Borel probability measure on $M$. $\mu$ is called conformal if all the positive Lyapunov exponents of ~$\mu$ are equal and the negative ones are equal respectively, that is $\lambda_1(\mu)=\cdots=\lambda_k(\mu)<0$ and $\lambda_{k+1}(\mu)=\cdots=\lambda_\ell(\mu)>0$. Then we say $f$ is an average conformal hyperbolic diffeomorphism on $M$ if every $f$-invariant ergodic measure is hyperbolic and conformal.
\end{enumerate}
\end{Definition}

\begin{Rem}
	By the result in~\cite{C2}, a map with every invariant ergodic measure expanding is actually uniformly expanding. However, a diffeomorphism with every invariant ergodic measure hyperbolic is not necessarily uniformly hyperbolic. The authors in ~\cite{CLR} constructed such an example.
\end{Rem}

\subsection{$(\chi, \varepsilon)$ repellers}

Let $M$ be a compact Riemannian manifold and $f:M\to M$ be a $C^1$ map.
For each $x\in M$, the following quantities
$$ \|D_xf\|=\sup_{0\neq u\in T_xM}\frac{\|D_xf(u)\|}{\|u\|}
\text{ and } m(D_xf)=\inf_{0\neq u\in T_xM}\frac{\|D_xf(u)\|}{\|u\|}
$$
are respectively called \emph{the maximal norm and minimum norm of the differentiable operator $D_xf: T_xM \to T_{fx}M$}, where $\|\cdot\|$ is the norm induced by the Riemannian metric on $M$.
We say that $W\subset M$ is \emph{forward invariant} if $f(W)\subset W$. A compact subset $W\subset M$ is said to be \emph{isolated} if there is a neighborhood $U$ of $W$ such that $W=\bigcap_{n\geq 0} f^{n}(U)$. Given a $f$-forward invariant subset $W\subset M$, we call $f|_W$ \emph{expanding} if there exists  $n \in \mathbb{N}$ such that for any $x\in W$ and any unit vector $v \in T_{x}M$, we have $\|D_{x}f^{n}(v)\|>1$. We say a compact $f$-forward invariant set $W$ is a \emph{repeller} if $W$ is isolated and expanding.

\begin{Definition}\label{defofacr}
A repeller $W$ is called average conformal if $f|_{W}$ is average conformal.
\end{Definition}

\begin{Definition} Given $\chi>0$ and  $\varepsilon \in(0, \chi)$, we say an average conformal repeller $W$ is a $(\chi,  \varepsilon)$ repeller  if for any $x \in W$ we have
\begin{displaymath}
\limsup_{n\rightarrow \infty}|\frac{1}{n}\log\|D_{x}f^{n}\|-\chi| < \varepsilon, \quad\quad\limsup_{n\rightarrow \infty}|\frac{1}{n}\log m(D_{x}f^{n})-\chi| < \varepsilon.
\end{displaymath}
\end{Definition}

\subsection{$(\chi^s, \chi^u, \varepsilon)$ horseshoes}

Let $M$ be a compact Riemannian manifold and $f:M\to M$ be a $C^1$ diffeomorphism.
We say that $W\subset M$ is \emph{$f$-invariant} if $f(W)=W$. A compact subset $W\subset M$ is said to be \emph{locally maximal (or isolated)} if there is a neighborhood $U$ of $W$ such that $W=\bigcap_{n\in\mathbb{Z}} f^{n}(U)$. $f$ is called \emph{topologically transitive} on $W$ if for any
two nonempty (relative) open sets $U_1$ and $U_2$ in $W$, we have $f^n(U_1)\cap U_2\neq\emptyset$ for some $n>0$.

\begin{Definition}
Let $W \subset M$ be a compact invariant set, we say that $W$ is a hyperbolic set if (up to a change of metric) there exists a $Df$-invariant splitting $T_{W}M = E^{s} \oplus E^{u}$ and number $\lambda>1$ such that for every $x \in W$ we have $$ \| D_xf|_{E^{s}_{x}} \|  \leq \lambda^{-1} < 1 < \lambda\leq \| D_xf|_{E^{u}_{x}} \|.$$
Here $\lambda^{-1}$ is called the skewness of the hyperbolicity.
\end{Definition}

\begin{Definition}
A hyperbolic set $W$ is called average conformal if it has two unique Lyapunov exponents, one negative and one positive. That is, for any invariant
ergodic measure $\mu$ on $W$, the Lyapunov exponents are $\lambda_1(\mu)=\cdots=\lambda_k(\mu)<0$ and $\lambda_{k+1}(\mu)=\cdots=\lambda_l(\mu)>0$ for some $0<k<l$.
\end{Definition}

A  set $W\subset M$ is called a {\it basic} set  if it is a compact, $f$-invariant, locally maximal hyperbolic set and $f|_{W}$ is topologically transitive. One can refer to  \cite{KH} for more details.

\begin{Definition}  Given constants $\chi^{s} < 0 < \chi^{u}$ and $\varepsilon \in (0, \min\{|\chi^{s}|,  \chi^{u}\})$, we say a basic set $W \subset M$ is a $(\chi^{s},  \chi^{u}, \varepsilon)$ horseshoe if for any $x\in W$ we have
\begin{displaymath}
\limsup_{|n|\rightarrow \infty}|\frac{1}{n}\log\|D_xf^{n}|_{E^{s}_{x}}\|-\chi^{s}|<\varepsilon,  \quad\quad     \limsup_{|n|\rightarrow \infty}|\frac{1}{n}\log\|D_xf^{n}|_{E^{u}_{x}}\|-\chi^{u}|<\varepsilon,
\end{displaymath}
\begin{displaymath}
\limsup_{|n|\rightarrow \infty}|\frac{1}{n}\log m(D_xf^{n}|_{E^{s}_{x}})-\chi^{s}|<\varepsilon,  \quad\quad     \limsup_{|n|\rightarrow \infty}|\frac{1}{n}\log m(D_xf^{n}|_{E^{u}_{x}})-\chi^{u}|<\varepsilon.
\end{displaymath}
\end{Definition}

Given an $f$-invariant ergodic measure $\mu$, which is hyperbolic and conformal, we denote the corresponding Lyapunov exponents by $\chi^s(\mu)$ and $\chi^u(\mu)$. Given $\varepsilon\in(0, \min\{|\chi^s(\mu)|, \chi^u(\mu)\})$,
we say a basic set $W \subset M$ is a \emph{$(\mu,  \varepsilon)$ horseshoe} if it is a $(\chi^{s}(\mu), ~\chi^{u}(\mu),  \varepsilon)$ horseshoe, which lies in the  $\varepsilon$ neighborhood of the support set of the measure $\mu$  and satisfies
$$|h(f, W)-h_\mu(f)|<\varepsilon.$$

\section{Proofs of Theorem \ref{main1}, Corollary \ref{dd} and Corollary \ref{urbanski}}

\subsection{Preparations}

In this section, we establish some technical results for Theorem \ref{main1}.

\subsubsection{Katok's theorem for $C^{1+\alpha}$ nonuniform expanding maps}
The following result is Katok's theorem for $C^{1+\alpha}$ nonuniform expanding maps. One can refer to \cite{PU} and \cite{QC} for proofs. Here we assume that $f$ is average conformal and hence the repeller we find is an average conformal one.

\begin{Thm}\label{Katok'repeller}
Let $f: M\rightarrow M$ be a $C^{1+\alpha}$ map of a compact Riemannian manifold $M$, and $\mu$ be an ergodic $f$-invariant Borel probability measure with positive entropy. Suppose that $f$ is average conformal expanding on $M$. Denote the Lyapunov exponent of $\mu$ by $\chi(\mu)>0$. Then for any $\varepsilon>0$, there exists a compact subset $W_{\varepsilon}\subset M$ satisfying
\begin{itemize}
\item [(a)] $W_{\varepsilon}$ is a $(\chi(\mu),\varepsilon)$ average conformal repeller. Moreover, there is a compact subset $R_{\varepsilon}\subset W_{\varepsilon}\subset M$ and a positive integer $N$ such that $f^{N}(R_{\varepsilon})=R_{\varepsilon}$. Here $f^{N}|_{R_{\varepsilon}}$ is expanding and topologically conjugate to a topologically mixing subshift of finite type and $W_{\varepsilon}=R_{\varepsilon}\cup f(R_\varepsilon)\cup\cdots\cup f^{N-1}(R_\varepsilon)$.
\item [(b)] $|h_{\mu}(f)-h(f|_{W_{\varepsilon}})|<\varepsilon$.
\item [(c)] $|\chi(\mu)-\chi(\nu)|<\varepsilon$ for every $\nu\in\mathcal{M}_{inv}(f|_{W_\varepsilon})$.
\end{itemize}
\end{Thm}

\subsubsection{Dimensional estimates for subsets of average conformal repellers}


\begin{Proposition}\label{relation}
Let $M$ be a $d$-dimensional compact ~Riemannian manifold and $f: M\rightarrow M$ be a $C^{1}$ map. Given~$0<\varepsilon<\chi$ and a $(\chi, \varepsilon)$ average conformal repeller $W\subset M$, for any subset $ Z \subset W$,  one has
$$\frac{h(f|_{W}, Z)}{\chi+\varepsilon}\leq \text{dim}_{H}Z\leq \frac{h(f|_{W}, Z)}{\chi-\varepsilon}.$$
\end{Proposition}

\begin{proof}
Since for any ~$n\in \mathbb{N}$ and any ~$x\in W$,
\begin{equation}\label{12}
m(D_xf^{n})\leq |\det (D_xf^{n})|^{\frac{1}{d}}\leq \|D_xf^{n}\|,
\end{equation}
and $W$ is a $(\chi, \varepsilon)$ average conformal repeller, then
 $$\limsup_{n\rightarrow \infty }\left|\frac{1}{n}\log|\det (D_xf^{n})|^{\frac{1}{d}}-\chi\right|<\varepsilon.$$
For any $K \in \mathbb{N}$, set
$$Z_{K}=\left\{x\in Z:  \left|\frac{1}{n}\log|\det (D_xf^{n})|^{\frac{1}{d}}-\chi\right|<\varepsilon \text{ for every } n\geq K\right\}.$$
It is easy to see  $Z=\bigcup_{K\in \mathbb{N}}Z_{K}$. For any fixed $K\in \mathbb{N}$, for any $x\in Z_{K}$ and $n\geq K$ we have
\begin{equation}\label{10}
e^{n(\chi-\varepsilon)}<|\det (D_xf^{n})|^{\frac{1}{d}}<e^{n(\chi+\varepsilon)}.
\end{equation}
Let $\phi(x) = \log |\det (D_xf)|^{\frac{1}{d}} $. Then
$\log|\det(D_xf^n)|^{\frac1d}=\sum_{i=0}^{n-1}\phi(f^ix)=S_n\phi(x)$. It yields that
$$n(\chi-\varepsilon) \leq S_n\phi(x) \leq n(\chi+\varepsilon) \text{ for every } x\in Z_K \text{ and } n\geq K.$$
Given $s \in \mathbb{R}, \ t \ge 0$,
$$m_{P}(Z_K, s, -t \phi, N, \delta)=\inf_{\mathcal{P}(Z_K, N, \delta)}\sum_{(x_{i}, n_{i})}\exp\left[-n_{i}s - t  S_{n_{i}} \phi(x_{i})\right].$$
It follows that
$$m_{P}(Z_K, s - t (\chi + \varepsilon ), 0, N, \delta) \le m_{P}(Z_K, s, -t \phi, N, \delta) \le m_{P}(Z_K, s -t (\chi - \varepsilon ), 0, N, \delta).$$
Thus
$$ m_{P}(Z_K, s - t  (\chi + \varepsilon ), 0, \delta) \le m_{P}(Z_K, s, -t \phi, \delta) \le m_{P}(Z_K, s - t (\chi - \varepsilon ), 0, \delta).$$
Therefore
$$     P_{Z_K}(0, \delta) - t (\chi + \varepsilon )  \le \  P_{Z_K}(-t \phi, \delta) \  \le \ P_{Z_K}(0, \delta) -t  (\chi - \varepsilon ). $$
Letting $\delta \to 0$, we obtain
$$  P_{Z_K}(0) - t  (\chi + \varepsilon )  \le     P_{Z_K}(-t \phi )   \le \ P_{Z_K}(0) -  t (\chi - \varepsilon ).$$
Theorem B in \cite{Y.Cao} tells us that
$$ P_{Z_K}(-\text{dim}_H Z_K \ \phi) =0.$$
Since  $ P_{Z_K}(0) = h(f|_W, Z_K) $,  then
$$\frac{h(f|_{W}, Z_K)}{\chi+\varepsilon}\leq \text{dim}_{H} Z_K\leq \frac{h(f|_{W}, Z_K)}{\chi-\varepsilon}.$$
Finally it follows from the properties of the entropy and Hausdorff dimension that
$$\frac{h(f|_{W}, Z)}{\chi+\varepsilon}\leq \text{dim}_{H} Z \leq \frac{h(f|_{W}, Z)}{\chi-\varepsilon}.$$
\end{proof}

\subsubsection{Dimensional approximations by repellers}

\begin{Proposition}\label{repeller}
Let $f: M\rightarrow M$ be a  $C^{1+\alpha}$ map of a $d$-dimensional compact Riemannian manifold $M$ and $\Lambda$ be a locally maximal compact invariant subset of $M$. Suppose $f$ is  average conformal expanding on $\Lambda$, then there exist a sequence of average conformal repellers $W_{n}\subset \Lambda$ and a sequence of expanding measures $\{\mu_{n}\}_{n>0}$ on $\Lambda$ such that
$$\lim_{n\rightarrow \infty }\text{dim}_{H} W_{n}=\lim_{n\rightarrow \infty }\text{dim}_{H} \mu_{n}=DD(f |_{\Lambda}).$$
\end{Proposition}

\begin{proof}
It follows from the definition of  $DD(f|_{\Lambda})$ that for any $n>0$, there exists $\mu_{n}\in\mathcal{M}_{erg}(f|_\Lambda)$ such that
$$h_{\mu_{n}}(f|_{\Lambda})>0 \mbox{ and } \text{dim}_{H} \mu_{n} \ge  DD(f|_{\Lambda}) - \frac{1}{n}.$$
Since $f$ is average conformal expanding on $\Lambda$, then $\mu_{n}$ is expanding and conformal, namely all the  Lyapunov exponents of $\mu_n$ are equal and we denote this value by $\chi(\mu_{n})>0$.
For any fixed $n\in\mathbb{N}$,
choose $\varepsilon_{n}\in\left(0, \min\left\{\frac{\chi(\mu_{n})}{n},  \frac{1}{n}\right\}\right)$.
Applying Theorem \ref{Katok'repeller} to $(\Lambda,f,\mu_n)$, one has that there exists  a $(\chi(\mu_{n}),  \varepsilon_{n})$ average conformal repeller $W_{n}\subset \Lambda$ such that
$$h_{\mu_{n}}(f|_{ \Lambda})-\varepsilon_{n}\leq h(f|_{W_{n}})\leq h_{\mu_{n}}(f|_{\Lambda})+\varepsilon_{n}$$
and $\left|\chi(\nu)-\chi(\mu_{n})\right|<\varepsilon_{n}$ for every $\nu\in\mathcal{M}_{erg}(f|_{W_n})$.
Applying Theorem B in \cite{Y.Cao} to  $(W_{n}, f)$, we have
$$P_{W_{n}}\left(f|_{W_{n}},-\text{dim}_{H}W_{n} \cdot \log \left|\det (Df|_{W_{n}})\right|^{\frac{1}{d}}\right) = 0.$$
Since $f$ is expanding on $W_{n}$, by the variational principle of topological entropy, there exists an ergodic measure $\nu_{n}\in\mathcal{M}_{erg}(f|_{W_n})$ such that $h_{\nu_{n}}(f|_{W_{n}})=h(f|_{W_{n}})$.
Thus
\begin{align*}
0&\geq h_{\nu_{n}}(f|_{W_{n}})-\int \text{dim}_{H} W_{n} \cdot \log |\det (Df|_{W_{n}})|^{\frac{1}{d}}d\nu_{n}\\
&=h_{\nu_{n}}(f|_{W_{n}})-\text{dim}_{H}W_{n} \cdot \chi(\nu_{n})\\
&= h(f|_{W_{n}})-\text{dim}_{H} W_{n} \cdot \chi(\nu_{n})\\
&\geq h_{\mu_{n}}(f|_{\Lambda})-\varepsilon_{n}-\text{dim}_{H} W_{n} \cdot \left(\chi(\mu_{n})+\varepsilon_{n}\right).
\end{align*}
Therefore $\text{dim}_{H} W_{n}\geq \frac{h_{\mu_{n}}(f|_{\Lambda})-\varepsilon_{n}}{\chi(\mu_{n})+\varepsilon_{n}}$.
Since $\mu_n$ is expanding and conformal, then $\text{dim}_{H} \mu_{n}=\frac{h_{\mu_{n}}(f|_{ \Lambda})}{\chi(\mu_{n})}$.
This implies that
\begin{equation}\label{compare1}
\begin{split}
\text{dim}_{H} W_{n}&\geq \text{dim}_{H} \mu_{n}\cdot\frac{\chi(\mu_{n})}{\chi(\mu_{n})+\varepsilon_{n}}-\frac{\varepsilon_{n}}{\chi(\mu_{n})+\varepsilon_{n}}\\
&\geq \left(DD(f|_{\Lambda})-\frac{1}{n}\right)\cdot\frac{\chi(\mu_{n})}{\chi(\mu_{n})+\varepsilon_{n}}-\frac{\varepsilon_{n}}{\chi(\mu_{n})+\varepsilon_{n}}\\
&\geq \left(DD(f|_{\Lambda})-\frac{1}{n}\right)\cdot\frac{1}{1+\frac{1}{n}}-\frac{1}{1+n}.
\end{split}
\end{equation}
On the other hand, it follows from Proposition \ref{relation} that
\begin{equation}\label{compare2}
\begin{split}
\text{dim}_{H} W_{n}&\leq \frac{h(f|_{W_{n}})}{\chi(\mu_{n})-\varepsilon_{n}}\\
&\leq \frac{h_{\mu_{n}}(f|_{\Lambda})+\varepsilon_{n}}{\chi(\mu_{n})-\varepsilon_{n}}\\
&=\text{dim}_{H}\mu_{n} \cdot\frac{\chi(\mu_{n})}{\chi(\mu_{n})-\varepsilon_{n}}+\frac{\varepsilon_{n}}{\chi(\mu_{n})-\varepsilon_{n}}\\
&\leq \text{dim}_{H}\mu_{n}\cdot\frac{1}{1-\frac{1}{n}}+\frac{1}{n-1}\\
&\leq DD(f|_{\Lambda})\cdot\frac{1}{1-\frac{1}{n}}+\frac{1}{n-1}.\\
\end{split}
\end{equation}
Combining with (\ref{compare1}), and letting  $n$ go to infinity, we get that $$\lim_{n\rightarrow \infty }\text{dim}_{H} W_{n}=\lim_{n\rightarrow \infty }\text{dim}_{H} \mu_{n}=DD(f|_{\Lambda}).$$
\end{proof}

The following Proposition is analogous. The proof of Proposition \ref{*} is parallel to that of Proposition \ref{repeller}.

\begin{Proposition}\label{*}
Let $f: M\rightarrow M$ be a  $C^{1+\alpha}$ map of a $d$-dimensional compact Riemannian manifold $M$ and $\mu$ be an expanding ergodic $f$-invariant Borel probability measure with positive entropy.
Suppose that $f$ is average conformal expanding on $M$. Then there exists  a sequence of average conformal repellers $\{W_{n}\}_{n\geq1}$ such that
$$\lim_{n\rightarrow \infty } \text{dim}_{H} W_{n}=\text{dim}_{H} \mu.$$
\end{Proposition}

\subsection{Proof of Theorem \ref{main1}}
(i) Since $f$ is average conformal expanding on $W$, then $\mu$ is expanding and conformal. We denote the Lyapunov exponent of $\mu$ by $\chi(\mu)$. Therefore
 \begin{equation}\label{Young formula}
 \text{dim}_{H} \mu = \frac{h_{\mu}(f|_{W})}{\chi(\mu)}.
 \end{equation}
Combining with $\text{dim}_{H} A< \text{dim}_{H} \mu$, one has
$$h_{\mu}(f|_{W}) > 0.$$
Katok's theorem tells us that for any $ \varepsilon >0 $, there exists an average conformal repeller $W_{\varepsilon}\subset W$ such that
  $$h_{\mu}(f|_{W})+\varepsilon \geq h(f|_{W_{\varepsilon}})\geq h_{\mu}(f|_{W})-\varepsilon  $$
and $|\chi(\nu)-\chi(\mu)|<\varepsilon$ for every $\nu\in\mathcal{M}_{erg}(f|_{W_\varepsilon})$.
It follows from Theorem B in \cite{Y.Cao} that
$$P(f|_{W_{\varepsilon}}, -\text{dim}_{H} W_{\varepsilon} \cdot \log|\det (Df)|^{\frac{1}{d}})=0.$$
Since $f$ is expanding on $W_\varepsilon$, by the variational principle for entropy, there exists $\nu\in\mathcal{M}_{erg}(f|_{W_\varepsilon})$ such that $h_{\nu}(f|_{W_{\varepsilon}})=h(f|_{W_{\varepsilon}})$.
Thus
\begin{align*}
0&\geq h(f|_{W_{\varepsilon}})-\text{dim}_{H} W_{\varepsilon} \cdot \chi(\nu)\\
&\geq h(f|_{W_{\varepsilon}})-\text{dim}_{H} W_{\varepsilon} \cdot (\chi(\mu)+\varepsilon)\\
&\geq h_{\mu}(f|_{W})-\varepsilon-\text{dim}_{H} W_{\varepsilon} \cdot (\chi(\mu)+\varepsilon).
\end{align*}
Therefore
$$\text{dim}_{H} W_{\varepsilon}\geq \frac{h_{\mu}(f|_{W})-\varepsilon}{\chi(\mu) + \varepsilon}.$$
Combining with (\ref{Young formula}), it has
\begin{equation}\label{star}
\lim\limits_{\varepsilon \rightarrow 0} \text{dim}_{H} W_{\varepsilon}  \geq  \text{dim}_{H} \mu.
\end{equation}
Since  $\text{dim}_{H} A < \text{dim}_{H} \mu$,  we can choose $\varepsilon>0$ small enough such that
$$\text{dim}_{H} A < \frac{\chi(\mu)-\varepsilon}{\chi(\mu)+\varepsilon} \cdot \text{dim}_{H} W_{\varepsilon}.$$
Since   $W_{\varepsilon}$ is a $(\chi(\mu), \varepsilon)$ average conformal repeller, by Proposition \ref{relation}, we have
$$\text{dim}_{H} A\geq \text{dim}_{H}(A\cap W_{\varepsilon}) \geq  \frac{h(f|_{W_{\varepsilon}}, A \cap W_{\varepsilon})}{\chi(\mu)+\varepsilon}$$
and
$$\text{dim}_{H} W_{\varepsilon} \leq \frac{h(f|_{W_{\varepsilon}})}{\chi(\mu)-\varepsilon}.$$
Hence
$$h(f|_{ W_{\varepsilon}}, A\cap  W_{\varepsilon})\leq (\chi(\mu)+\varepsilon)\cdot\text{dim}_{H} A < (\chi(\mu)-\varepsilon)\cdot\text{dim}_{H} W_{\varepsilon} \leq h(f|_{W_{\varepsilon}}).$$
Since $W_\varepsilon$ is $f$-invariant, then $E^{+}_{f|_{W_{\varepsilon}}}(A\cap W_{\varepsilon}) = E^{+}_{f|_{W_{\varepsilon}}}(A)$. Therefore
\begin{eqnarray*}
\begin{aligned}
h(f|_{W}, E^{+}_{f|_{W}}(A))&\geq h(f|_{ W_{\varepsilon}}, E^{+}_{f|_{W_{\varepsilon}}}(A))\\
 &=h(f|_{ W_{\varepsilon}}, E^{+}_{f|_{W_{\varepsilon}}}(A\cap W_\varepsilon))\\
 &= h(f|_{W_{\varepsilon}})\\
 &\geq h_{\mu}(f|_{W})-\varepsilon,
\end{aligned}
\end{eqnarray*}
the second equality is by Proposition $5.1$ in \cite{CG1}.
The arbitrariness of $\varepsilon>0$ implies that
$$h(f|_{W}, E^{+}_{f |_{W}}(A))\geq h_{\mu}(f |_{W}).$$

Since $W_\varepsilon$ is $f$-invariant and $W_\varepsilon\subset W$, then $E^{+}_{f|_{W_{\varepsilon}}}(A\cap W_{\varepsilon})\subset E^{+}_{f|_{W}}(A)$. Therefore
\begeq\label{dimension1}
\dim_{H} E^{+}_{f|_{W_{\varepsilon}}}(A\cap W_{\varepsilon})\leq \dim_{H} E^{+}_{f|_{W}}(A).\\
\eeq
Then we have
\begin{eqnarray}\label{dimension2}
\begin{aligned}
\dim_{H}~E^{+}_{f|_{W_{\varepsilon}}}(A\cap W_{\varepsilon})&\geq \frac{h(f|_{ W_{\varepsilon}}, E^{+}_{f|_{W_{\varepsilon}}}(A\cap W_{\varepsilon}))}{h(f|_{W_{\varepsilon}})}\cdot
\frac{\chi(\mu)-\varepsilon}{\chi(\mu)+\varepsilon} \cdot \dim_{H} W_\varepsilon\\
&=\frac{\chi(\mu)-\varepsilon}{\chi(\mu)+\varepsilon} \cdot \dim_{H} W_\varepsilon,
\end{aligned}
\end{eqnarray}
here the above inequality is by Proposition \ref{relation}, and the above equality is by Proposition $5.1$ in \cite{CG1}.
Thus (\ref{dimension1}) and (\ref{dimension2}) 
imply
$$\dim_H E^{+}_{f|_{W}}(A) \ge \text{dim}_{H} E^{+}_{f|_{W_\varepsilon}}(A\cap W_\varepsilon)\geq \frac{\chi(\mu)-\varepsilon}{\chi(\mu)+\varepsilon} \cdot \text{dim}_{H} W_\varepsilon.$$
Letting $\varepsilon$ tend to $0$ and combining with (\ref{star}), one has
$$\text{dim}_{H}E^{+}_{f|_{W}}(A)\geq \dim_H\mu.$$
This completes the proof of (i) in Theorem \ref{main1}.

(ii)Since $f$ is average conformal expanding on $W$ and $h(f|_W, A)<h_\mu(f|_W)$, then $\mu$ is
conformal and expanding with $h_{\mu}(f|_{W}) > 0$. Hence we denote the Lyapunov exponent of $\mu$ by $\chi(\mu)$.
Then
Katok's theorem and (\ref{star}) tell us that for any small  $\varepsilon > 0$, we can find an appropriate $(\chi(\mu),  \varepsilon)$ average conformal repeller  $W_{\varepsilon}\subset W$ such that
$$\dim_{H} W_\varepsilon\geq \text{dim}_{H} \mu-\varepsilon\quad \text{and}\quad
h(f|_{W}, A) < h(f|_{W_{\varepsilon}}).$$
Hence
\begin{equation}\label{star2}
h(f|_{ W_{\varepsilon}}, A\cap  W_{\varepsilon})\le  h(f|_{W}, A) < h(f|_{W_{\varepsilon}}).
\end{equation}
Since $E^{+}_{f|_{W_{\varepsilon}}}(A\cap W_{\varepsilon})\subset W_\varepsilon$,
it follows that
$$\frac{\text{dim}_{H}E^{+}_{f|_{W_{\varepsilon}}}(A\cap W_{\varepsilon})}{\text{dim}_{H}W_\varepsilon}\geq \frac{h(f|_{ W_{\varepsilon}},   E^{+}_{f|_{W_{\varepsilon}}}(A\cap W_{\varepsilon}))}{h(f|_{W_{\varepsilon}})}\cdot\frac{\chi(\mu)-\varepsilon}{\chi(\mu)+\varepsilon}=\frac{\chi(\mu)-\varepsilon}{\chi(\mu)+\varepsilon},$$
here the inequality is by Proposition \ref{relation} and the equality is by (\ref{star2}) and Proposition $5.1$ in \cite{CG1}.
Thus it has
$$\frac{\chi(\mu)+\varepsilon}{\chi(\mu)-\varepsilon} \cdot \text{dim}_{H} E^{+}_{f|_{W_{\varepsilon}}}(A\cap W_{\varepsilon})\geq \text{dim}_{H} W_\varepsilon\geq \text{dim}_{H}~\mu-\varepsilon.$$
Since $W_\varepsilon$ is $f$-invariant and $W_\varepsilon\subset W$, then $E^{+}_{f|_{W_{\varepsilon}}}(A\cap W_{\varepsilon})\subset E^{+}_{f|_{W}}(A)$.
It yields that ~$\text{dim}_{H}E^{+}_{f|_{W_{\varepsilon}}}(A\cap W_{\varepsilon})\leq \text{dim}_{H} E^{+}_{f|_{W}}(A)$.
Therefore it has
$$\frac{\chi(\mu)+\varepsilon}{\chi(\mu)-\varepsilon} \cdot \text{dim}_{H}E^{+}_{f|_{W}}(A)\geq \text{dim}_{H} \mu-\varepsilon.$$
By the arbitrariness of $\varepsilon>0$, we obtain
$$\text{dim}_{H} E^{+}_{f|_{W}}(A)\geq \text{dim}_{H} \mu.$$
This completes the proof of (ii) in Theorem \ref{main1}.

\subsection{Proof of Corollary \ref{dd}}
It follows from the proof of Proposition \ref{repeller} that there exist a sequence of expanding measures $\{\mu_n\}_{n\geq1}$ on $W$ and a sequence of $(\chi(\mu_{n}), \varepsilon_{n})$ average conformal repellers
$\{W_{n}\}_{n\geq1}$ contained in $W$ such that
$$\lim_{n\rightarrow \infty } \text{dim}_{H} W_{n}=\lim_{n\to\infty}\text{dim}_H\mu_n=DD(f|_{W}),$$
where $\chi(\mu_n)$ is the Lyapunov exponent of $\mu_n$, and $\varepsilon_n\in\left(0, \min\{\frac{\chi(\mu_n)}{n}, \frac1n\}\right)$. Thus for $n>0$ large enough, one has
$$\dim_HA < \dim_H \mu_n \leq DD(f|_W).$$
By Theorem \ref{main1}, we obtain
$$\dim_H E^+_{f|_W}(A) \geq \dim_H \mu_n.$$
Let $n\to\infty$, one gets $\dim_H E^+_{f|_W}(A) \geq DD(f|_W).$
This completes the proof of Corollary \ref{dd}.

\subsection{Proof of Corollary \ref{urbanski}}

Since $W$ is an average conformal repeller with respect to $f$, there exists an expanding ergodic measure $\mu$ with $ h_{\mu}(f |_W) > 0$ such that

\begin{equation*}
\text{dim}_H \mu = \text{dim}_H W. \footnote{The equality is an implied conclusion of Theorem C in \cite{Y.Cao}.}
\end{equation*}
Then combining with Corollary \ref{dd} one has
$$\text{dim}_H E^{+}_{f|_{W}}(A) = \text{dim}_H W,$$
which proves Corollary \ref{urbanski}.

\section{Proofs of Theorem \ref{main2} and Corollary \ref{ddd}}

\subsection{Preparations}

In this section, we establish some technical results for Theorem \ref{main2}.

\subsubsection{Katok's Theorem for $C^{1+\alpha}$ nonuniform hyperbolic diffeomorphisms}

\begin{Proposition}\label{horseshoe}
Let $f: M \rightarrow M$ be a $C^{1+\alpha}$ diffeomorphism of a $d$-dimensional compact Riemannian manifold $M$, and $\mu$ be an ergodic $f$-invariant Borel probability measure with positive measure-theoretical entropy $h_{\mu}(f)>0$. Suppose that $f$ is average conformal hyperbolic on M. Denote the Lyapunov exponents of $\mu$ by $\chi^u(\mu)>0$ and $\chi^s(\mu)<0$. Then for any $\varepsilon>0$, there exist $\delta(\varepsilon)>0$ and a compact subset $\Gamma=\Gamma(\mu,\varepsilon)$ of $M$ such that
\begin{itemize}
\item [(a)] $\Gamma$ is a $(\mu,  \varepsilon)$ average conformal horseshoe. Moreover, there is a positive integer $N$ and a compact subset $R\subset \Gamma$ such that
 $f^N(R)=R$. Here
 $f^{N}|_{R}$ is hyperbolic, topologically conjugate to a full two-sided shift in the symbolic space and $\Gamma=\bigcup_{i=1}^{N}f^{i}(R)$,
\item [(b)] $h(f, \Gamma)\geq h_{\mu}(f)-\varepsilon$,
\item [(c)] $\text{dim}_{H}(\Gamma\cap W_{loc}^{\ast}(x, f))\geq \frac{h_{\mu}(f)}{|\chi^{\ast}(\mu)|}-\delta(\varepsilon)$, where $\ast\in\{s, u\}$.
\end{itemize}
\end{Proposition}

\begin{proof}
Theorem $1$ in \cite{Katok2} and $f$ is average conformal hyperbolic, imply ($a$), ($b$), and
\begin{equation*}
|\chi^u(\mu)-\chi^u(\nu)|<\varepsilon \text{ and } |\chi^s(\mu)-\chi^s(\nu)|<\varepsilon \text{ for every } \nu\in\mathcal{M}_{erg}(f|_\Gamma).
\end{equation*}
Denote $\tau = \text{dim}_{H}(\Gamma \cap W^{u}_{\text{loc}}(x, f))$.
It follows from Remark $3.2$ in \cite{WWCZ} that
\begin{displaymath}
\sup_{\nu \in \mathcal{M}_{erg}(f|_{\Gamma})}\left\{h_{\nu}(f|_{\Gamma})- \tau\lim_{n\rightarrow \infty}\int \frac{1}{n}\log m(Df^{n}|_{E^{u}_{x}})d\nu\right\}=0.
\end{displaymath}
Since $f$ is hyperbolic on $\Gamma$, by the variational principle for entropy, there exists $\nu\in\mathcal{M}_{erg}(f|_\Gamma)$ such that $h_{\nu}(f|_{\Gamma})= h(f, \Gamma)$.
It follows that
\begin{equation*}
\begin{split}
0&\geq h_{\nu}(f|_{\Gamma})-\tau\lim_{n\rightarrow \infty}\int \frac{1}{n}\log m(Df^{n}|_{E^{u}_{x}})d\nu \\
&\geq h_{\nu}(f|_{\Gamma})-\tau\lim_{n\rightarrow \infty}\int \frac{1}{n}\log |\det Df^{n}|_{E^{u}_{x}}|^{\frac{1}{\text{dim} E^{u}_{x}}}d\nu\\
&=h_{\nu}(f|_{\Gamma})-\tau\chi^u(\nu)\\
&\geq h(f, \Gamma)-\tau\cdot(\chi^u(\mu)+\varepsilon).
\end{split}
\end{equation*}
Thus
$$\text{dim}_{H}(\Gamma \cap W^{u}_{\text{loc}}(x, f))\geq \frac{h(f, \Gamma)}{\chi^u(\mu)+\varepsilon}\geq\frac{h_\mu(f)-\varepsilon}{\chi^u(\mu)+\varepsilon},$$
which implies the desired result ($c$) for $\ast=u$.
Using the same arguments, one can prove the lower bound for $\text{dim}_{H}(\Gamma \cap W_{\text{loc}}^{s}(x, f))$.
This completes the proof of Proposition \ref{horseshoe}.
\end{proof}

\subsubsection{Bowen equation for s-saturated invariant subsets of the average conformal horseshoes}


We first give the definition of s-saturated set as follow.

\begin{Definition}
Given a hyperbolic set $\Gamma \subset M$, we say a set $Z \subset \Gamma$ is  s-saturated if for any  $x\in Z$, we have  $\Gamma \cap W_{\text{loc}}^{s}(x, f)\subset Z$.
\end{Definition}

Cao proved in \cite{Y.Cao} that for any subset $Z$ of an average conformal repeller for  a $C^{1}$ map, the Hausdorff dimension of the set $Z$ is the unique root of some Bowen equation. We exploit his idea in an essential way to prove that for any  $s$-{\it saturated} invariant set $Z$ of an average conformal hyperbolic set of a $C^{1}$ diffeomorphism, the Hausdorff dimension of the intersection of the set $Z$ with the unstable leaf is the unique root of the corresponding Bowen equation.  If $Z = \Lambda$,  the authors in \cite{WWCZ} proved that the Hausdorff dimension of the set  $ \Lambda  \cap  W^{u}_{\text{loc}}(x, f) \ $ equals to  the unique root of  Bowen equation $P_{\Lambda }(-t \log \phi(x))=0$. Since $\Lambda$ is compact and invariant,  we can use the  variational principle for topological pressure to obtain the lower bound estimate. But  for a general $s$-{\it saturated} invariant set $Z$, we don't have variational principle for topological pressure. Therefore, in this case, the proof that
the Hausdorff dimension of the set $ Z \cap W^{u}_{\text{loc}}(x, f)$ equals to  the unique root of  Bowen equation $P_{Z }(-t \log \phi(x))=0$ is different.

\begin{Proposition}\label{Bowen2}
Let $f:  M\rightarrow M$ be a  $C^{1}$ diffeomorphism of a compact Riemannian manifold $M$, and $\Lambda\subset M$ be an average conformal hyperbolic set. Suppose $f$ is topologically transitive on $\Lambda$. Denote $d_{1}=\text{dim}\ E^{u}_{x}$ and $\phi(x)= |\det (D_xf|_{E^{u}_{x}})|^{\frac{1}{d_{1}}}$. Then for any  $s$-{\it saturated} invariant subset  $Z\subset \Lambda$ and any $x\in Z$, one has
$$\text{dim}_{H}(Z\cap W^{u}_{\text{loc}}(x, f))=t^{\ast},$$
where $t^{\ast}$ is the unique root of  $P_{Z}(f, -t\log \phi)=0$.
\end{Proposition}

\begin{proof}
First of all we prove that  $\text{dim}_{H}(Z\cap W^{u}_{\text{loc}}(x, f))\geq t^{\ast}$ for every $x\in Z$. Since it holds for $t^{\ast}=0$ automatically, then we suppose $t^{\ast}>0$. For  $0 < s <t^{\ast}$,
we know $P_{Z}(f, -s\log \phi) > 0$.
Let $\rho>0$ be  small enough  such that $2s\rho < P_{Z}(f, -s\log \phi)$. It follows from Lemma $2.1$ in \cite{WWCZ} that there exists $N_{0}\in \mathbb{N}$ such that for  $n\geq N_{0}$ and any  $x\in \Lambda$,
\begin{equation}\label{star4}
1\leq \frac{\prod_{i=0}^{n-1}\phi(f^{i}(x))}{m(D_xf^{n}|_{E^{u}_{x}})}\leq e^{n\rho} \text{ and } e^{-n\rho}\leq \frac{\prod_{i=0}^{n-1}\phi(f^{i}(x))}{||D_xf^{n}|_{E^{u}_{x}}||}\leq 1.
\end{equation}
For the above $\rho>0$, there exists $\delta>0$ such that for any $x,y\in M$ with $d(x,y)< \delta$,
$$e^{-N_{0}\rho}< \frac{\|D_xf^{N_{0}}|_{E^{u}_{x}}\|}{\|D_yf^{N_{0}}|_{E^{u}_{y}}\|}< e^{N_{0}\rho}.$$

For any fixed $x\in Z$, since $f$ is topologically transitive on $\Lambda$ and $Z$ is $s$-saturated, then there exists $m\in\mathbb{N}$ such that for any $y\in Z$,
\begin{equation}\label{star6}
\left(\bigcup_{i=0}^{m} f^i(W_{\text{loc}}^u(x, f))\right)\cap W^{s}_{\text{loc}}(y, f)\cap Z \neq \varnothing.
\end{equation}
By the definition of the $s$-dimensional Hausdorff measure, there exists $\varepsilon_{0}>0$ such that for any $\varepsilon\in(0, \varepsilon_0)$, one has an open cover $$\mathcal{C}_{\varepsilon}=\left\{B_{u}(x_{i}, r_{i}): x_i\in Z\cap \left(\bigcup_{i=0}^mf^i(W^{u}_{\text{loc}}(x,f))\right) \text{ and } r_{i}\in(0, \varepsilon)\right\}$$
of the set $Z\cap \left(\bigcup_{i=0}^mf^i(W^{u}_{\text{loc}}(x,f))\right)$  and
$$\mathcal{H}^{s}_{\varepsilon}\left(Z\cap \left(\bigcup_{i=0}^mf^i(W^{u}_{\text{loc}}(x,f))\right)\right) + 1\geq \sum_{B_{u}(x_{i}, r_{i})\in \mathcal{C}_{\varepsilon}}(2r_{i})^{s}.$$
(Here $B_{u}(x_i,r_i)=\{y\in W^{u}_{\text{loc}}(x_i,f): d_{u}(x_i,y)<r_i\}$, and $d_{u}$ is the distance on the unstable leaf induced by the Riemannian metric.)
For any ball $B_{u}(x_{i}, r_{i})\in\mathcal{C}_\varepsilon$,  there exists $n_{i}\in \mathbb{N}$ such that
\begin{equation}\label{star5}
B_{u}(x_{i}, r_{i})\subset B_{n_{i}}(x_{i}, \delta), \ \mbox{ but} \ \  B_{u}(x_{i}, r_{i})\nsubseteq B_{n_{i}+1}(x_{i}, \delta).
\end{equation}
We can choose the above $\varepsilon_{0}$ sufficiently small such that $n_{i}> N_{0}$ for all $i$. For each $n_{i}$, there exist $q\in \mathbb{N}$ and $0\leq r<N_{0}$ such that $n_{i}=qN_{0}+r$. Let
$$K=\max_{0\leq r<N_{0}}\max\limits_{y\in M}\|D_yf^{r}|_{E^{u}_{y}}\| \, \mbox{ and } \  L=\min_{0\leq r<N_{0}}\min_{y\in M}\prod_{j=0}^{r-1}\phi(f^{j}(y)).$$
(\ref{star5}) implies that there is $y\in B_{u}(x_{i}, r_{i})$ such that $ d(f^{j}(x_{i}), f^{j}(y)) < \delta$ for  $j=0, 1, \cdots,  n_{i}-1$, but $d(f^{n_{i}}(x_{i}), f^{n_{i}}(y))\geq \delta$. Hence there exists $z\in B_{u}(x_{i},r_{i})$ with $d(f^{j}(x_{i}), f^{j}(z)) < \delta$  for  $j=0, 1,  \cdots,  n_{i}-1$,  and
 \begin{align*}
\delta &\leq d(f^{n_{i}}(x_{i}), f^{n_{i}}(y))\\
&\leq d_{u}(f^{n_{i}}(x_{i}), f^{n_{i}}(y))\\
&\leq \|D_zf^{n_{i}}|_{E^{u}_{z}}\|d_{u}(x_{i}, y)\\
&\leq K\left(\prod_{j=0}^{q-1}\|D_{f^{jN_0}(z)}f^{N_{0}}|_{E^{u}_{f^{jN_{0}}(z)}}\|\right) \cdot d_{u}(x_{i}, y)\\
&< e^{qN_{0}\rho}K\left(\prod_{j=0}^{q-1}\|D_{f^{jN_0}(x_i)}f^{N_{0}}|_{E^{u}_{f^{jN_{0}}(x_{i})}}\|\right) \cdot d_{u}(x_{i}, y)\\
&< Ke^{2qN_{0}\rho}r_{i}\prod_{j=0}^{qN_{0}-1}\phi(f^{j}(x_{i})),
\end{align*}
here the last inequality is by (\ref{star4}).
It follows that
\begin{align*}
r_{i}&> \frac{\delta}{Ke^{2qN_{0}\rho}\cdot \prod_{j=0}^{qN_{0}-1}\phi(f^{j}(x_{i}))}\\
&\geq\frac{\delta \left(\prod^{n_{i}-1}_{j=qN_{0}}\phi(f^{j}(x_{i}))\right) \cdot e^{2\rho r}}{K\left(\prod_{j=0}^{n_{i}-1}\phi(f^{j}(x_{i}))\right) \cdot e^{2n_{i}\rho}}\\
&\geq\frac{\delta L}{K\left(\prod_{j=0}^{n_{i}-1}\phi(f^{j}(x_{i}))\right) \cdot e^{2n_{i}\rho}}.
\end{align*}
Set
 \begin{equation*}
 \mathcal{C}_{\varepsilon}^{\ast}=\{B_{u}^{\ast}(x_{i}, r_{i}): B_{u}^{\ast}(x_{i}, r_{i})=\bigcup_{p\in B_{u}(x_{i}, r_{i})}W^{s}_{\text{loc}}(p, f) \text{ for each } B_u(x_i, r_i)\in\mathcal{C}_\varepsilon\}.
 \end{equation*}
It follows from (\ref{star6}) that $\mathcal{C}_{\varepsilon}^{\ast}$ is a cover of the set $Z$. Therefore there is a positive constant $a$ depending only on the dimension of the manifold such that for each $B_{u}^{\ast}(x_{i}, r_{i})\in \mathcal{C}_{\varepsilon}^{\ast}$,
$$B_{u}^{\ast}(x_{i}, r_{i})\subset B_{n_{i}}(x_{i}, a \delta).$$
It follows that
\begin{align*}
\sum_{B_{u}(x_{i}, r_{i})\in \mathcal{C}_{\varepsilon}}(2r_{i})^{s}&> (\frac{2\delta L}{K})^{s} \sum_{B_{u}(x_{i}, r_{i})\in \mathcal{C}_{\varepsilon}}\frac{1}{\left(\prod_{j=0}^{n_{i}-1}(\phi(f^{j}(x_{i})))^{s}\right) \cdot e^{2sn_{i}\rho}}\\
&=(\frac{2\delta L}{K})^{s}\sum_{(x_{i}, n_{i})} \exp\left[-S_{n_{i}}(s\log \phi(x_{i}))+(-2sn_{i}\rho)\right]\\
&\geq (\frac{2\delta L}{K})^{s}m_{P}(Z, 2s\rho, -s\log \phi, N, a\delta),
\end{align*}
where  $N=\min\{n_{i}\}$. It is easy to see if  $\varepsilon \rightarrow 0$, then $N\rightarrow \infty$.
Since  $2s\rho < P_{Z}(f, -s\log \phi)$, then
 $$\lim_{N\rightarrow \infty}m_{P}(Z, 2s\rho, -s\log \phi, N, a\delta)=\infty.$$
It implies that
$$\lim_{\varepsilon \rightarrow 0}\mathcal{H}^{s}_{\varepsilon}\left(Z\cap \left(\bigcup_{i=0}^mf^{i}(W^{u}_{\text{loc}}(x, f))\right)\right)=\infty.$$
Hence
$$\text{dim}_{H}\left(Z \cap \left(\bigcup_{i=0}^mf^{i}(W^{u}_{\text{loc}}(x, f))\right)\right)\geq s.$$
The arbitrariness  of $s < t^{\ast} $  implies that
$$\text{dim}_{H}\left(Z \cap  \left(\bigcup_{i=0}^mf^{i}(W^{u}_{\text{loc}}(x, f))\right)\right)\geq t^{\ast}.$$
Since $\text{dim}_H\left(Z\cap \left(\bigcup_{i=0}^mf^{i}(W^{u}_{\text{loc}}(x, f))\right)\right)=\max_{0\leq i\leq m}\text{dim}_H(Z\cap f^i(W_{\text{loc}}^u(x,f)))$
and  the invariance of the set $Z$, we obtain
$$\text{dim}_{H}  (Z \cap W^{u}_{\text{loc}}(x, f) ) = \text{dim}_{H}\left(Z\cap \left(\bigcup_{i=0}^mf^{i}(W^{u}_{\text{loc}}(x, f))\right)\right) \geq t^{\ast}.$$

On the other hand we will prove $\text{dim}_{H}(Z \cap W^{u}_{\text{loc}}(x, f))\leq t^{\ast}$ for every $x\in Z$. For any  $s > t^{\ast}$, we have  $P_{Z}(f, -s \log \phi) < 0$ and denote  $A = P_{Z}(f, -s\log \phi)$. Let $\rho > 0$ be small enough such that  $6s \rho < -A$. Since  $A<0$, there exists $\delta_{0}>0$  such that for any $\delta\in(0,\delta_{0})$, it has
$$P_{Z}(-s\log \phi, \delta)<\frac{A}{2} < 0.$$
Thus
$$\lim_{N\rightarrow \infty} m_{P}(Z, \frac{A}{2}, -s\log \phi, N, \delta)=0,$$
where
$$m_{P}(Z, \frac{A}{2}, -s\log \phi, N, \delta)=\inf\left\{\sum_{B_{n_i}(x_i, \delta)\in \mathcal{C}} \exp\left[-\frac{A}{2}n_{i}-sS_{n_{i}}(\log \phi(x_{i}))\right]\right\},$$
the infimum is taken over all the covers $\mathcal{C}=\{B_{n_{i}}(x_{i}, \delta):x_{i}\in Z \text{ and } n_{i}\geq N\}$ of the set $Z$.
For the above $\rho>0$, there is $0< \delta <\delta_{0}$ such that for any $x,y\in M$ with $d(x,y) < 2\delta$, we have
\begin{align}\label{star7}
|\phi(x)-\phi(y)|< \rho \mbox{ and }  e^{-N_{0}\rho}<\frac{m(D_xf^{N_{0}}|_{E^{u}_{x}})}{m(D_yf^{N_{0}}|_{E^{u}_{y}})}<e^{N_{0}\rho}.
\end{align}
For any fixed $x\in Z$ and $N>N_0$, then $n_i>N_0$ for each $n_i$ in $\mathcal{C}$. Set
\begin{equation*}
\overline{\mathcal{C}}=\{B_{n_i}(x_i, \delta)\in\mathcal{C}: B_{n_i}(x_i, \delta)\cap Z\cap W^u_{\text{loc}}(x, f)\neq\emptyset\}
\end{equation*}
and
\begin{equation*}
\mathcal{C}'=\{B_{n_i}(x_i', 2\delta): B_{n_i}(x_i, \delta)\in\overline{\mathcal{C}} \text{ and } x_{i}^{\prime} \in B_{n_{i}}(x_{i}, \delta) \cap Z\cap W^{u}_{\text{loc}}(x, f)\}.
\end{equation*}
It is easy to see $\mathcal{C}'$ is  a cover of the set $Z\cap W^{u}_{\text{loc}}(x, f)$.
For each $n_{i}$, there exist $q\in \mathbb{N}$ and $0\leq r<N_{0}$ such that $n_{i}-1=qN_{0}+r$.
Notice that there is $\kappa>1$ such that for any $x\in \Lambda$ and  $y, z\in W^{u}_{\text{loc}}(x,f)$, we have  $d_{u}(y, z)\leq \kappa d(y, z)$.
Denote
$$L^{\prime}=\min_{0\leq r<N_{0}}\min_{y\in M} \left\{m(D_yf^{r})\right\} \  \mbox { and } \  \  K^{\prime}=\max_{0\leq r<N_{0}}\max_{y\in M} \left\{\prod_{j=0}^{r-1}\phi(f^{j}(y))\right\}.$$
For any   $y\in B_{n_{i}}(x_{i}^{\prime}, 2\delta)\cap W^{u}_{\text{loc}}(x, f)$,  there exists $\xi\in B_{n_{i}}(x_{i}^{\prime}, 2\delta)$ such that
\begin{align*}
d_{u}(f^{n_{i}-1}x_{i}^{\prime}, f^{n_{i}-1}y) & \ge  \  m(D_\xi f^{n_{i}-1}|_{E^{u}_{\xi}})d_{u}(x_{i}^{\prime}, y)\\
&  \ge \  \prod^{q-1}_{j=0}m(D_{f^{jN_0}(\xi)}f^{N_{0}}|_{E^{u}_{f^{jN_{0}}(\xi)}})\cdot m(D_{f^{qN_0}(\xi)}f^{r}|_{E^{u}_{f^{qN_{0}}(\xi)}})\cdot d_{u}(x_{i}^{\prime}, y)\\
& \ge \  e^{-qN_{0}\rho}L^{\prime}\cdot \prod^{q-1}_{j=0}m(D_{f^{jN_0}(x_i')}f^{N_{0}}|_{E^{u}_{f^{jN_{0}}(x^{\prime}_{i})}})\cdot d_{u}(x_{i}^{\prime}, y)\\
& \ge  \ e^{-2qN_{0}\rho} L^{\prime}\cdot \prod_{j=0}^{qN_{0}-1}\phi(f^{j}(x^{\prime}_{i}))\cdot d_{u}(x_{i}^{\prime}, y).
\end{align*}
Since
\begin{align*}
d_{u}(f^{n_{i}-1}x_{i}^{\prime}, f^{n_{i}-1}y) \leq \kappa d(f^{n_{i}-1}x_{i}^{\prime}, f^{n_{i}-1}y)<2\delta\kappa,
\end{align*}
then  it has
$$e^{-2n_{i}\rho} \prod_{j=0}^{n_{i}-1}\phi (f^{j}(x_{i}^{\prime}))\cdot d_{u}(x_{i}^{\prime}, y)<\frac{2\delta\kappa K^{\prime}}{L^{\prime}}.$$
Hence we get an estimate of the diameter of  $B_{n_{i}}(x^{\prime}_{i}, 2\delta)\cap W^{u}_{\text{loc}}(x, f)$:
\begin{equation}\label{star8}
\text{diam}(B_{n_{i}}(x_{i}^{\prime}, 2\delta)\cap W^{u}_{\text{loc}}(x, f))\leq \frac{4\delta\kappa K^{\prime} e^{2n_{i}\rho}}{L^{\prime}\prod_{j=0}^{n_{i}-1}\phi (f^{j}(x_{i}^{\prime}))}\leq \frac{4\delta\kappa K^{\prime} e^{2n_{i}\rho}}{L^{\prime}\lambda^{n_{i}}}\triangleq l_{n_{i}},
\end{equation}
where  $\lambda>1$ is the inverse of the skewness from the hyperbolicity. It is easy to see that  $\lim_{n\rightarrow \infty} l_{n}=0$ for sufficiently small $\rho>0$.
It follows from (\ref{star7}) and (\ref{star8}) that
\begin{align*}
&m_{P}(Z, \frac{A}{2}, -s\log \phi, N, 2\delta)\\
=&\inf_{\mathcal{C}}\{\sum_{B_{n_i}(x_i, \delta)\in \mathcal{C}} \exp[-\frac{A}{2}n_{i}-sS_{n_{i}}(\log \phi(x_{i}))]\}\\
\geq& \inf_{\mathcal{C}}\{\sum_{B_{n_{i}}(x_{i}^{\prime}, 2\delta)\cap W^{u}_{\text{loc}}(x, f)\in \mathcal{C}^{\prime}} \exp[-\frac{A}{2}n_{i}-n_{i}\rho s-sS_{n_{i}}(\log \phi(x^{\prime}_{i}))]\}\\
\geq& \inf_{\mathcal{C}}\{\sum_{B_{n_{i}}(x_{i}^{\prime}, 2\delta)\cap W^{u}_{\text{loc}}(x, f)\in \mathcal{C}^{\prime}} \exp(-\frac{A}{2}n_{i}-3n_{i}\rho s)\cdot (\frac{L^{\prime}}{4\delta\kappa K^{\prime}})^{s}\\
&\ \ \ \ \ \ \ \ \ \ \ \ \ \ \ \ \ \ \ \ \ \ \ \ \ \ \ \ \ \ \ \ \cdot (\text{diam} (B_{n_{i}}(x_{i}^{\prime}, 2\delta)\cap W^{u}_{\text{loc}}(x, f)))^{s}\}\\
\geq& \inf_{\mathcal{U}}\{\sum_{B_{i}\in \mathcal{U}} (\frac{L^{\prime}}{4\delta\kappa K^{\prime}})^{s}\cdot (\text{diam} B_{i})^{s}\}\\
=& (\frac{L^{\prime}}{4\delta\kappa K^{\prime}})^{s}\cdot \mathcal{H}^{s}_{l_{N}}(Z\cap W^{u}_{\text{loc}}(x, f)),
\end{align*}
where the infimum is taken over all the finite or countable cover $\mathcal{U}$ of  $Z\cap W^{u}_{\text{loc}}(x, f)$ satisfying  $\text{diam}(\mathcal{U})<l_{N}$.
Letting  $N$ tend to infinity, we get that $$\mathcal{H}^{s}(Z\cap W^{u}_{\text{loc}}(x, f))=0.$$
Thus $\text{dim}_{H}(Z \cap W^{u}_{\text{loc}}(x, f))\leq s $ for every $s > t^{\ast } $. Since this holds for all $s > t^{\ast}$, we have
  $$\text{dim}_{H}(Z \cap W^{u}_{\text{loc}}(x, f))\leq t^{\ast}.$$
\end{proof}

\subsubsection{Dimensional estimates of the intersection of the s-{\it saturated} invariant subsets of average conformal horseshoes with the unstable manifold}

By Proposition \ref{Bowen2}, we prove a parallel result to Proposition \ref{relation} for s-{\it saturated} invariant subsets of an average conformal horseshoes.

\begin{Proposition}\label{relation2}
Let  $M$ be a compact Riemannian manifold and $f: M\rightarrow M$ be a $C^{1+\alpha}$ diffeomorphism. Given $\chi^{s}<0<\chi^{u}$ and $0<\varepsilon<\min\{\chi^{u}, -\chi^{s}\}$, let $\Gamma$ be a  $ (\chi^{s}, \chi^{u}, \varepsilon)$ horseshoe. Suppose $\Gamma$ is average conformal and topologically transitive, then for any $s$-saturated invariant subset $Z \subset \Gamma$ and any  $x\in Z$, one has
$$\frac{h(f|_{\Gamma}, Z)}{\chi^{u}+\varepsilon}\leq \text{dim}_{H}(Z\cap W^{u}_{\text{loc}}(x, f))\leq \frac{h(f|_{\Gamma}, Z)}{\chi^{u}-\varepsilon}.$$
\end{Proposition}

\begin{proof}
Since for $x\in Z$,
\begin{displaymath}
\limsup_{n\rightarrow \infty}\left|\frac{1}{n}\log\|D_xf^{n}|_{E^{u}_{x}}\|-\chi^{u}\right|<\varepsilon \text{ and } \limsup_{n\rightarrow \infty}\left|\frac{1}{n}\log m(D_xf^{n}|_{E^{u}_{x}})-\chi^{u}\right|<\varepsilon,
\end{displaymath}
then
\begin{displaymath}
\limsup_{n\rightarrow \infty}\left|\frac{1}{n}\log|\det (D_xf^{n}|_{E^{u}_{x}})|^{\frac{1}{d_{1}}}-\chi^{u}\right|<\varepsilon,
\end{displaymath}
where $d_1=\text{dim} E_x^u$. For any $K \in \mathbb{N}$, set
$$Z_{K}=\left\{x\in Z:\left|\frac{1}{n}\log|\det (D_xf^{n}|_{E^u_x})|^{\frac{1}{d_1}}-\chi^u \right|<\varepsilon \text{ for every } n\geq K\right\}.$$
It is easy to see  $Z=\bigcup_{K\in \mathbb{N}}Z_{K}$. For any fixed $K\in \mathbb{N}$, for any $x\in Z_{K}$ and  $n\geq K$,
\begin{equation}\label{10}
e^{n(\chi^u -\varepsilon)}<|\det (D_xf^{n}|_{E^u_x})|^{\frac{1}{d_1}}<e^{n(\chi^u+\varepsilon)}.
\end{equation}

Let $\phi(x) = \log |\det (D_xf |_{E^u_x}) | ^{\frac{1}{d_1}} $. Next, we can give a similar  estimate of $P_{Z_K}(f, -t \phi) $  as in  the proof of Proposition \ref{relation}. It follows that
$$  P_{Z_K}(f, 0) - t  (\chi^u + \varepsilon )  \le     P_{Z_K}(f, -t \phi )   \le \ P_{Z_K}(f, 0) -  t (\chi^u - \varepsilon ).$$
Proposition \ref{Bowen2} tells us that $P_{Z_K}(f, - t^*  \phi) =0$, here $t^* = \text{dim}_H (Z_K \cap W_{\text{loc}}^u(x,f))$.
Since  $ P_{Z_K}(f, 0) = h(f|_{\Gamma}, Z_K) $, then it follows that
$$\frac{h(f|_{\Gamma}, Z_K)}{\chi^u +\varepsilon}\leq t^* \leq \frac{h(f|_{\Gamma}, Z_K)}{\chi^u -\varepsilon}.$$
By the properties of the entropy and Hausdorff dimension, we complete the proof of Proposition \ref{relation2}.
\end{proof}

\subsubsection{Dimensional approximations by average conformal horseshoes}

The following proposition for diffeomorphisms is parallel to Proposition \ref{repeller} for maps.

\begin{Proposition}\label{ha}
Let $f: M\rightarrow M$ be a $C^{1+\alpha}$ diffeomorphism on a compact Riemannian manifold $M$ and $W\subset M$ be a locally maximal compact $f$-invariant set. Suppose $f$ is average conformal hyperbolic on $W$. Then there exist a sequence of hyperbolic ergodic measures $\{\mu_{n}\}_{n>0}$ on $W$  and a sequence of average conformal hyperbolic sets  $\Gamma_{n} \subset  W$ such that
\begin{displaymath}
\lim_{n\rightarrow \infty}\text{dim}_{H} \Gamma_{n}=\lim_{n\rightarrow \infty}\text{dim}_{H} \mu_{n}=DD(f|_{W}).
\end{displaymath}
Moreover, for any positive integer $n$, there is some $\varepsilon_n>0$ such that $\Gamma_n$ is $(\chi^s(\mu_{n}), \chi^u(\mu_{n})$, $\varepsilon_{n})$ horseshoe,
where $\chi^s(\mu_n)<0$ and $\chi^u(\mu_n)>0$ are the Lyapunov exponents of $\mu_n$ respectively.
\end{Proposition}

\begin{proof}
By the definition of  $DD(f|_{W})$ and $f$ is average conformal hyperbolic on $W$, there exist a sequence of hyperbolic ergodic measures  $\{\mu_{n}\}_{n>0}$ such that $h_{\mu_{n}}(f)>0$ and
$$\lim_{n\rightarrow \infty} \text{dim}_{H} \mu_{n} = DD(f|_{W}).$$
Choose $\varepsilon_n\in\left(0, \min\{-\chi^{s}(\mu_{n}), \chi^{u}(\mu_{n}),\frac{1}{n}\}\right)$. It follows from Proposition \ref{horseshoe} that there exist a sequence of  $(\chi^s(\mu_{n}),  \ \chi^u(\mu_{n}), \ \varepsilon_{n})$ average conformal horseshoes  $\{\Gamma_{n}\}_{n}$ contained in $W$ such that
$$h_{\mu_{n}}(f|_{W})-\varepsilon_{n}\leq h(f|_{\Gamma_{n}})\leq h_{\mu_{n}}(f|_{W})+\varepsilon_{n},$$
and $| \chi^u(\mu_n) - \chi^u(\nu_n) |  < \varepsilon_n, \, \mbox{ and } \,  | \chi^s(\mu_n) - \chi^s(\nu_n) | < \varepsilon_n$ for every $\nu_n\in\mathcal{M}_{erg}(f|_{\Gamma_n})$.
Therefore Proposition \ref{relation2} tells us that
$$\frac{h(f|_{\Gamma_n})}{\chi^u(\mu_n) +\varepsilon_n}\leq \text{dim}_H (\Gamma_n \cap W_{\text{loc}}^u(x,f)) \leq \frac{h(f|_{\Gamma_n})}{\chi^u(\mu_n) -\varepsilon_n}.$$
Similarly, we can prove that
$$ \text{dim}_H (\Gamma_n \cap W_{\text{loc}}^s(x,f)) \leq \frac{h(f|_{\Gamma_n})}{-\chi^s(\mu_n) -\varepsilon_n}.
$$
By Theorem $A$ in \cite{WWCZ}, it has
$$ \text{dim}_{H} \Gamma_{n} =\text{dim}_{H}(W^{u}_{\text{loc}}(x, f)\cap \Gamma_{n}) +\text{dim}_{H}(W^{s}_{\text{loc}}(x, f)\cap \Gamma_{n}).$$
Therefore
\begin{align*}
 \text{dim}_{H} \Gamma_{n} &\leq \frac{h(f|_{\Gamma_n})}{\chi^{u}(\mu_{n})-\varepsilon_{n}}+\frac{h(f|_{\Gamma_n})}{-\chi^{s}(\mu_{n})-\varepsilon_{n}}\\
 & \leq \frac{h_{\mu_n}(f|_{W}) + \varepsilon}{\chi^{u}(\mu_{n})-\varepsilon_{n}}+\frac{h_{\mu_n}(f|_{W}) + \varepsilon}{-\chi^{s}(\mu_{n})-\varepsilon_{n}}.
 \end{align*}
On the other hand, it follows from Proposition \ref{horseshoe} that
$$
\text{dim}_{H} \Gamma_{n} \geq \frac{h_{\mu_{n}}(f|_{W})}{-\chi^{s}(\mu_{n})}+\frac{h_{\mu_{n}}(f|_{W})}{\chi^{u}(\mu_{n})}-2\delta(\varepsilon_{n})=\text{dim}_{H}\mu_{n}-2\delta(\varepsilon_{n}).$$
Letting  $n$ tend to infinity, one has
$$\lim_{n\rightarrow \infty}\text{dim}_{H} \Gamma_{n}=  \lim_{n\rightarrow \infty}\text{dim}_{H} \mu_{n} = DD(f|_{W}),$$
which implies the desired result.
\end{proof}

\begin{Rem}\label{rem3}
Actually, we can modify the proof above slightly to prove that given a hyperbolic ergodic measure $\mu$ with positive entropy, there is a sequence of average conformal hyperbolic sets $\Gamma_{n}$ such that
\begin{align*}
\lim_{n\rightarrow \infty} \text{dim}_{H} \Gamma_{n}=\text{dim}_{H} \mu.
\end{align*}
\end{Rem}

\subsubsection{Entropy on the limit exceptional set}

For each $p\in \mathbb{N}$, denote the space of two-sided infinite sequences of $p$ symbols by
$$\Sigma_{p}=\{\underline i=(\cdots i_{-1}i_0i_1 \cdots)\}=\{1, 2, ..., p\}^\mathbb{Z}.$$
We endow the space $\Sigma_p$ with the metric
$$d(\underline i, \underline i')=\sum_{k=-\infty}^{+\infty}\frac{|i_k - i_k'|}{2^{|k|}},$$
where $\underline i=(\cdots i_{-1}i_0i_1\cdots)$ and $\underline i'=(\cdots i_{-1}'i_0'i_1'\cdots)$. It induces the compact topology on $\Sigma_p$ with cylinders to be disjoint open subsets. The two-sided shift $\sigma: \Sigma_{p} \rightarrow \Sigma_{p}$ is defined by $\sigma(\underline i)_k=i_{k+1}$. 
We denote by $\Sigma_p^+$ and $\Sigma_p^-$ respectively the sets of right-sided and left-sided infinite sequences obtained from $\Sigma_p$. We consider the one-sided shifts $\sigma^+: \Sigma_p^+\to \Sigma_p^+$ and $\sigma^-: \Sigma_p^- \to \Sigma_p^-$ defined by
$$\sigma^+(i_0i_1\cdots)=(i_1i_2\cdots) \text{ and } \sigma^-(\cdots i_{-1}i_{0})=(\cdots i_{-2}i_{-1}).$$
Let $\pi^+: \Sigma_p\to\Sigma_p^+$ and $\pi^-: \Sigma_p\to\Sigma_p^-$ be the projections defined respectively by
$$\pi^+(\cdots i_{-1}i_0i_1\cdots) = (i_0i_1\cdots) \text{ and } \pi^-(\cdots i_{-1}i_0i_1\cdots) = (\cdots i_{-1}i_0).$$

Given $n, m\in\mathbb{Z}$ and a finite sequence $(i_ni_{n+1}\cdots i_m)$, a cylinder of length $m-n+1$ is defined by
$$[i_n\cdots i_m]=\{\underline j=(\cdots j_{-1}j_0j_1\cdots)\in\Sigma_p: j_k=i_k \text{ for } k=n,...,m\}.$$
Given a positive integer $n$, let $\Sigma^{+}_{p, n}=\{1, 2, ..., p\}^{n}$. Denote $<U>:=n$ for $U\in \Sigma^{+}_{p, n}$.
Given  $\mathcal{U}\subset \bigcup_{n\geq 1}\Sigma^{+}_{p, n}$, denote
$$I^{+}(\mathcal{U})=\{\underline{i}\in \Sigma^{+}_{p}: \text{ for every } n, m \in\mathbb{N}\cup\{0\} \text{ with } n<m, \text{ we have } (i_{n}\cdots i_{m})\notin \mathcal{U}\}.$$
The following two results were proved by Dolgopyat in~\cite{Do}, which are used in the proof of Proposition \ref{key}.

\begin{Thm}\label{Do}\cite{Do}~
If a set $A\subset \Sigma_{p}^{+}$ satisfies  $h(\sigma^{+}, A)<h(\sigma^{+})$, then
$$h(\sigma^{+}, I^{+}_{\sigma^{+}|\Sigma^{+}_{p}}(A))=h(\sigma^{+}).$$
\end{Thm}

\begin{Proposition}\label{Do2}\cite{Do}~
Suppose there is a function ~$H:\mathbb{N}\rightarrow \mathbb{R}$  which satisfies
$$\lim_{n\rightarrow \infty}H(n)=h(\sigma^{+}),$$
 and if for  some $s\in \left(0, h(\sigma^{+})\right)$, $n_{0}\geq 1$, there exists $\mathcal{U}=\{U_{\textit{l}}:U_{\textit{l}}\in \bigcup_{n\geq n_{0}}\Sigma^{+}_{p, n} \}$ with $\sum_{\textit{l}}e^{-s<U_{\textit{l}}>}<1$, then we get
 $$h\left(\sigma^{+}, I^{+}(\mathcal{U})\right)\geq H(n_{0}).$$
\end{Proposition}


Consider a $C^{1+\alpha}$ diffeomorphism $f: M \to M$ of a compact Riemannian manifold $M$. We assume $f$ is average conformal hyperbolic on  a locally maximal compact $f$-invariant subset $W\subset M$. Let $\mu$ be an ergodic $f$-invariant probability measure on $W$ and $A$ be a subset of $W$ such that $\dim_H A < \dim_H\mu$.
For any sufficiently small $\varepsilon>0$,
Proposition~\ref{horseshoe} tells us that there exists a ~$(\chi^s, \chi^u, \varepsilon)$ horseshoe~$\Gamma_{\varepsilon}\subset W$.
The following proposition proves the limit $A\cap \Gamma_\varepsilon$-exceptional set for $f|_{\Gamma_\varepsilon}$ has full topological entropy.
It is the key proposition in the proof of Theorem \ref{main2},
which also extends Campos and Gelfert's result in ~\cite{CG2} to non-conformal case.

\begin{Proposition}\label{key}
Let $f: M\rightarrow M$ be a  $C^{1+\alpha}$ diffeomorphism on a $d$-dimensional compact Riemannian manifold $M$, and $W\subset M$ be a locally maximal compact $f$-invariant set. Suppose $f$ is average conformal hyperbolic on $W$. Let $\mu\in \mathcal{M}_{erg}(f|_{W})$. If  $A\subset W$ satisfies $\dim_{H}A < \dim_{H} \mu$,
then there exists $\varepsilon_{0}>0$ such that for any ~$\varepsilon\in (0, ~\varepsilon_{0})$, and any ~$(\mu, ~\varepsilon)$ horseshoe ~$\Gamma=\Gamma_{\varepsilon}\subset W$, we have ~$$h(f|_{\Gamma}, I^{+}_{f|_{\Gamma}}(A\cap \Gamma))=h(f|_{\Gamma}).$$
\end{Proposition}

\begin{proof}
Let $\chi^s$ and $\chi^u$ be the negative and positive Lyapunov exponents of $\mu$ respectively. Wang and Cao \cite{WC} proved that
\begin{equation}\label{LEformular}
\dim_H\mu=h_\mu(f)\left(\frac{1}{\chi^u}-\frac{1}{\chi^s}\right).
\end{equation}
Since $\dim_{H}\mu > 0$, then  $h_{\mu}(f)>0$. Choose $\theta>0$ satisfying ~$\dim_{H}A<\theta<\dim_{H}\mu$. Denote
\begin{equation}\label{delta}
\delta\triangleq 1-\frac{\theta}{\dim_{H}\mu}
\end{equation}
and fix some $\varepsilon_1>0$ small enough such that we have
\begin{eqnarray}\label{varepsilon}
\begin{aligned}
&\frac{2\varepsilon_1}{|\chi^s|}<\frac\delta4,\quad \frac{2\varepsilon_1\theta \left(1+\frac{2\varepsilon_1}{\min\{\chi^u, -\chi^s\} - \varepsilon_1}\right)}{h_\mu(f)}<\frac{\delta^2}{8},\quad \varepsilon_1<h_\mu(f)\frac{\delta^2}{8}\quad \text{and }\\
&(1-\delta)\left(1+\frac{2\varepsilon_1}{\min\{\chi^u, -\chi^s\} - \varepsilon_1}\right) <1-\frac\delta2.
\end{aligned}
\end{eqnarray}
Take $\varepsilon_0=\min\left\{\varepsilon_1, \frac{h_\mu(f)}{2}\right\}$. 

Given $\varepsilon\in(0, \varepsilon_0)$. By Proposition~\ref{horseshoe},
there exists a ~$(\chi^s, \chi^u, \varepsilon)$ horseshoe~$\Gamma_{\varepsilon}\subset W$. Let~$d_{1}=\dim~E^{s}_{x}$ and $d_{2}=\dim~E^{u}_{x}$.
Denote
$$\psi^{-}(x)\triangleq \log|\det (D_xf|_{E^{s}_{x}})|^{\frac{1}{d_{1}}}, \quad \psi^{+}(x)\triangleq -\log|\det (D_xf|_{E^{u}_{x}})|^{\frac{1}{d_{2}}}.$$
By the continuity of ~$\psi^{-}(x)$ and ~$\psi^{+}(x)$ and  the compactness of $\Gamma_{\varepsilon}$,  there exists a positive number ~$C_{0}$ such that
\begeq\label{1}
-C_{0}\leq \min\{\psi^{-}, \psi^{+}\}.
\eeq
For every positive integer $n\geq1$, denote
$$S_{-n}\psi^-=\psi^- \circ f^{-1}+\psi^- \circ f^{-2}+...+\psi^- \circ f^{-n}$$
and
$$S_{n}\psi^+=\psi^+ +\psi^+ \circ f+...+\psi^+ \circ f^{n-1}.$$
Wang, Wang, Cao and Zhao \cite{WWCZ} proved that
\begin{equation}\label{8}
\lim_{n\rightarrow \infty}\frac{1}{n}\left(\log||D_xf^{n}|_{E^{s}_{x}}||-\log m(D_xf^{n}|_{E^{s}_{x}})\right)=0
\end{equation}
and
\begin{equation}\label{10}
\lim_{n\rightarrow \infty}\frac{1}{n}\left(\log||D_xf^{n}|_{E^{u}_{x}}||-\log m(D_xf^{n}|_{E^{u}_{x}})\right)=0
\end{equation}
uniformly for ~$x\in \Gamma_{\varepsilon}$.
By the hyperbolicity of $f|_{\Gamma_\varepsilon}$, the definition of the ~$(\chi^{s}, \chi^{u}, \varepsilon)$ horseshoe, (\ref{8}) and (\ref{10}), one has there exists~$N_{0}=N_{0}(\varepsilon)\geq 1$ such that for any ~$x\in \Gamma_{\varepsilon}$ and any ~$n\geq N_{0}$, we have
\begeq\label{2}
\left|\frac{1}{n}S_{-n}\psi^{-}(x)-\chi^{s}(\mu)\right|<2\varepsilon,\quad\quad \left|\frac{1}{n}S_{n}\psi^{+}(x)+\chi^{u}(\mu)\right|<2\varepsilon,
\eeq
\begin{align}\label{9}
1\leq \frac{\|D_xf^{n}|_{E^{u}_{x}}\|}{|\det (D_xf^{n}|_{E^{u}_{x}})|^{\frac{1}{d_{2}}}}\leq e^{n\varepsilon}, \quad\quad 1\leq \frac{\|D_xf^{n}|_{E^{s}_{x}}\|}{|\det (D_xf^{n}|_{E^{s}_{x}})|^{\frac{1}{d_{1}}}}\leq e^{n\varepsilon},
\end{align}
and
\begin{align}\label{smallmol}
1\leq \frac{|\det (D_xf^{n}|_{E^{u}_{x}})|^{\frac{1}{d_{2}}}}{m(D_xf^{n}|_{E^{u}_{x}})}\leq e^{n\varepsilon}, \quad\quad 1\leq \frac{|\det (D_xf^{n}|_{E^{s}_{x}})|^{\frac{1}{d_{1}}}}{m(D_xf^{n}|_{E^{s}_{x}})}\leq e^{n\varepsilon}.
\end{align}
We also require that
\begin{equation}\label{N0}
\frac{C_0}{|\chi^s| \cdot 2N_0}<\frac{\delta}{4}.
\end{equation}
For the above $\varepsilon>0$ and $N_0=N_0(\varepsilon)$, there exists $\delta>0$ such that
\begin{equation}\label{continuousofdf}
e^{-N_0\varepsilon} \leq \frac{\|D_yf^{N_0}|_{E^i_y}\|}{\|D_zf^{N_0}|_{E^i_z}\|} \leq e^{N_0\varepsilon}, \quad e^{-N_0\varepsilon} \leq \frac{m(D_yf^{N_0}|_{E^i_y})}{m(D_zf^{N_0}|_{E^i_z})} \leq e^{N_0\varepsilon}
\end{equation}
for every $x\in\Gamma_\varepsilon$, every $y, z\in W^u_{\text{loc}}(x, f^{N_0})$ with $d_u(y, z)\leq\delta$ and $i\in\{u, s\}$.
Notice that $f|_{\Gamma_{\varepsilon}}$ is conjugate to the shift ~$\sigma|_{\Sigma_{p}}$ for some ~$p\in \mathbb{N}$, i.e. there exists a homeomorphism  $\pi:\Sigma_{p}\rightarrow \Gamma_{\varepsilon}$
such that~$\pi\circ \sigma=f|_{\Gamma_{\varepsilon}}\circ \pi$.
Note that
\begin{equation}\label{measurableentropy}
h(\sigma, \Sigma_p)=h(f, \Gamma_\varepsilon)\geq h_\mu(f)-\varepsilon.
\end{equation}
Let $\mathcal{P}=\{R_1, ..., R_p\}$ be a Markov partition of $\Gamma_\varepsilon$ (with respect to $f$). We also assume the diameter of the Markov partition is smaller than $\delta$.

Fix $r\in (0, 1)$ small enough. Since ~$\psi^{+}(x)$ is always negative and uniformly bounded, we know that for any ~$\underline{i}\in \Sigma_{p}$, there exists a minimal positive integer ~$n=n^+(\underline{i}, r)$ such that
\begin{equation}\label{estimationofn}
S_{n}\psi^{+}(\pi(\underline{i}))<\log r\leq S_{n-1}\psi^{+}(\pi(\underline{i})).
\end{equation}
Combining with (\ref{1}), one has
\begin{equation}\label{estimationofpsi+}
S_{n}\psi^{+}(\pi(\underline{i}))<\log r\leq S_{n}\psi^{+}(\pi(\underline{i}))+C_{0}.
\end{equation}
Since ~$\psi^{+}$ is continuous, we know that $n^+(\underline{i}, r)$ is locally constant. Thus by the compactness of $\Sigma_{p}$, there exist integers~$N_{1}^{+}(r)\leq N_{2}^{+}(r)$, depending only on ~$r$, such that for any ~$\underline{i} \in \Sigma_{p}$, we have ~$N_{1}^{+}(r)\leq n^+(\underline{i}, r)\leq N_{2}^{+}(r)$.
Given $n\in\{N_1^+(r), N_1^+(r)+1, \cdots, N_2^+(r)\}$, let
\begin{eqnarray*}
\begin{aligned}
\tilde{C}^+_n(r)=\{[.i_0i_1\cdots i_{n-1}]\subset \Sigma_p: &\text{ there is some } \underline{i}\in[.i_0i_1\cdots i_{n-1}]\\
 &\text{ such that } n^+(\underline{i}, r)=n\}
\end{aligned}
\end{eqnarray*}
and $l_n\triangleq \text{ card } \tilde{C}^+_n(r)$. Analogously we consider the potential function $\psi^-$. For every $\underline{i}\in\Sigma_M$, let $m=m^-(\underline{i}, r)\geq 1$ be the positive integer such that
\begin{equation}\label{estimationofm}
S_{-m}\psi^-(\pi(\underline{i}))<\log r \leq S_{-(m-1)}\psi^-(\pi(\underline{i})).
\end{equation}
Again $m^-(\underline{i}, r)$ is well defined and (\ref{1}) implies
\begin{equation}\label{estimationofpsi-}
S_{-m}\psi^{-}(\pi(\underline{i}))<\log r\leq S_{-m}\psi^{-}(\pi(\underline{i}))+C_{0}.
\end{equation}

We first define a partition $C^+(r)$ of $\Sigma_p^+$ recursively. We start with index $n=N_1^+$. Let
\begin{equation*}
S^+(N_1^+)=\Sigma_p \quad \text{ and }\quad  C_{N_1^+}^+(r)=\tilde{C}_{N_1^+}^+(r).
\end{equation*}
Assuming that all objects are already defined for each $k=N_1^+, \cdots, n$ and $S^+(n)\neq\varnothing$, let
\begin{equation*}
S^+(n+1)\triangleq S^+(n)\setminus \{C: C\in\tilde{C}_{n}^+(r)\}
\end{equation*}
and
\begin{equation*}
C^+_{n+1}(r)\triangleq \{C: C\in\tilde{C}_{n+1}^+(r) \text{ and } C\subseteq S^+(n+1)\}.
\end{equation*}
Since $n^{+}(\underbar{\textit{i}},r)\leq N_2^+$ for each $\underbar{\textit{i}}\in\Sigma_{p}$, we eventually arrive at $S^+(n^*)=\varnothing$ for some $n^*\leq N_2^+$. Then we stop this recursion and define the family
\begin{equation*}
C^+(r)\triangleq \{\pi^+(C): C\in C_n^+(r) \text{ for } n=N_1^+, \cdots, n^*\}.
\end{equation*}
Therefore $C^+(r)$ partitions $\Sigma_p^+$ into pairwise disjoint cylinders. We construct analogously $C^-(r)$ wich partitions $\Sigma_p^-$ into pairwise disjoint cylinders. Denote by $N_1^-=N_1^-(r)$ and $N_2^-=N_2^-(r)$ the correspondingly defined positive integers.

Finally, concatenating the symbolic sequence cylinders, define
\begin{eqnarray*}
\begin{aligned}
C(r)\triangleq \{[i_{-m}\cdots i_{-1}.i_0i_1 \cdots i_{n-1}]:\ &[i_{-m}\cdots i_{-2}i_{-1}.]\in C^-(r) \text{ and }\\
 &[.i_0i_1\cdots i_{n-1}]\in C^+(r)\},
\end{aligned}
\end{eqnarray*}
which partitions $\Sigma_p$ into pairwise disjoint cylinders. Set
$$N_1(r)\triangleq \min\{N_1^-(r), N_1^+(r)\}\quad \text{ and }\quad N_2(r)\triangleq \min\{N_2^-(r), N_2^+(r)\}.$$
Note that $C(r)$ has the following properties:
\begin{itemize}
\item [(a)] It is a family of pairwise disjoint cylinders, each of which is of length between $2N_1(r)$ and $2N_2(r)$.
\item [(b)] Each cylinder of length $k=m+n$ in $C(r)$ contains a sequence $\underline{i} \in \Sigma_p$ satisfying $n^+(\underline{i}, r)$, $n^-(\underline{i}, r)\in\{N_1(r), \cdots, N_2(r)\}$.
\end{itemize}

Given $\rho>0$ put
$$\mathcal{C}_\rho\triangleq \bigcup_{r\in(0, \rho)}C(r).$$
The conjugation map $\pi: \Sigma_p\to\Gamma_\varepsilon$ sends each cylinder $C=[i_{-m}\cdots i_{-1}.i_0i_1\cdots i_{n-1}]\in C(r)$ into a Markov rectangle
$$R_{i_{-m}\cdots i_{n-1}}=\pi(C)=\bigcap_{k=-m}^{n-1}f^{-k}(R_{i_k}).$$
Denote $\mathcal{R}_\rho\triangleq \displaystyle\bigcup_{r\in(0, \rho)}R(r)$, where $R(r)=\{\pi(C): C\in C(r)\}$.

\begin{Claim}\label{propertyofcover}
The family $\mathcal{R}_\rho$ satisfies the following properties:
\begin{itemize}
  \item [(P1)] We have $\varnothing\in \mathcal{R}_\rho$ and $|R|>0$ for every nonempty $R\in \mathcal{R}_\rho$, where $| \cdot |$ denotes the diameter of a set.
  \item [(P2)] For every $\varepsilon>0$, there exists a finite or countable subcollection $\mathcal{F}'\subset \mathcal{R}_\rho$ such that $\Gamma_\varepsilon\subset \displaystyle\bigcup_{R\in\mathcal{F}'} R$ and $|R|\leq\varepsilon$ for every $R\in\mathcal{F}'$.
  \item [(P3)] There exist positive constants $\kappa_1$, $\kappa_2$ and $\kappa_3$ such that for every $R\in \mathcal{R}_\rho$, we can choose $r\in(0, \rho)$ satisfying $R\in R(r)$ and
  $$\kappa_1 \cdot r^{1+\kappa_3\varepsilon} \leq |R| \leq \kappa_2 \cdot r^{1-\kappa_3\varepsilon}.$$
\end{itemize}
\end{Claim}

\begin{proof}
The conclusions (P1) and (P2) are obvious. We only prove (P3).
For every $R\in \mathcal{R}_\rho$, there exist $r\in(0, \rho)$ and $C=[i_{-m}\cdots i_{-1}.i_0i_1\cdots i_{n-1}]\in C(r)$ such that $R=\pi(C)=\displaystyle\bigcap_{k=-m}^{n-1}f^{-k}(R_{i_k})$. It follows from ($b$) that we can choose some $\underline{i}\in C$ such that $n^+(\underline{i}, r)=n$ and $n^-(\underline{i}, r)=m$.
Let
\begin{eqnarray*}
\begin{aligned}
\tilde{\kappa}_0&=\displaystyle\min_{z\in M}\left\{m(D_zf), m(D_zf^2), ..., m(D_zf^{N_0-1})\right\},\\
\tilde{\kappa}_1&=\displaystyle\max_{z\in\Gamma_\varepsilon}\left\{(|\det(D_zf|_{E^s_z})|^{\frac{1}{d_1}})^{-1}, |\det(D_zf|_{E^u_z})|^{\frac{1}{d_2}}\right\},\\
\tilde{\kappa}_2&=\displaystyle\min_{z\in\Gamma_\varepsilon}\left\{(|\det(D_zf|_{E^s_z})|^{\frac{1}{d_1}})^{-1}, |\det(D_zf|_{E^u_z})|^{\frac{1}{d_2}}\right\}.
\end{aligned}
\end{eqnarray*}
Denote $x=\pi(\underline{i})$.

For every $y\in R$, by (\ref{continuousofdf}), (\ref{smallmol}) and (\ref{estimationofn}), we obtain there is $\xi\in W^u_{\text{loc}}(x, f)$ such that
\begin{eqnarray}\label{unstablees}
\begin{aligned}
d_u(f^{n-1}x, f^{n-1}y)&\geq m(D_\xi f^{n-1}|_{E^u_\xi}) \cdot d_u(x,y)\\
&\geq \prod_{k=0}^{p-1} m(D_{f^{kN_0}\xi}f^{N_0}|_{E^u_{f^{kN_0}\xi}}) \cdot \tilde{\kappa}_0 \cdot d_u(x, y)\\
&\geq \prod_{k=0}^{p-1} m(D_{f^{kN_0}x}f^{N_0}|_{E^u_{f^{kN_0}x}}) \cdot e^{-n\varepsilon} \cdot \tilde{\kappa}_0 \cdot d_u(x, y)\\
&\geq \prod_{k=0}^{p-1} |\det(D_{f^{kN_0}x}f^{N_0}|_{E^u_{f^{kN_0}x}})|^{\frac{1}{d_2}} \cdot e^{-2n\varepsilon} \cdot \tilde{\kappa}_0 \cdot d_u(x, y)\\
&\geq |\det (D_xf^n|_{E^u_x})|^{\frac{1}{d_2}} \cdot e^{-2n\varepsilon} \cdot \tilde{\kappa}_0\tilde{\kappa}_1^{-N_0} \cdot d_u(x, y)\\
&= \tilde{\kappa} e^{-S_n\psi^+(x)} \cdot e^{-2n\varepsilon} \cdot d_u(x, y)\\
&\geq \tilde{\kappa} e^{-\log r} \cdot e^{-2n\varepsilon} \cdot d_u(x, y),
\end{aligned}
\end{eqnarray}
where $d_u$ is the metric induced by the Riemannian structure on the unstable manifold, $n-1=pN_0+q$ and $0\leq q<N_0$, $\tilde{\kappa}=\tilde{\kappa}_0\tilde{\kappa}_1^{-N_0}$.
It follows from (\ref{estimationofn}) that $n\leq -\frac{1}{\log\tilde{\kappa}_2}\log r+1$. Therefore combining with (\ref{unstablees}) one has
\begin{eqnarray*}
\begin{aligned}
d_u(f^{n-1}x, f^{n-1}y)&\geq \tilde{\kappa} r^{-1} e^{\frac{2\varepsilon}{\log\tilde{\kappa}_2}\log r - 2\varepsilon} \cdot d_u(x,y)\\
&\geq \kappa' \cdot r^{-1+\frac{2\varepsilon}{\log\tilde{\kappa}_2}} \cdot d_u(x,y).
\end{aligned}
\end{eqnarray*}
Thus
$$d_u(x, y)\leq (\kappa')^{-1} \cdot r^{1-\frac{2}{\log\tilde{\kappa}_2}\varepsilon} \cdot d_u(f^{n-1}x, f^{n-1}y).$$

For every $y\in R$, by (\ref{continuousofdf}), (\ref{smallmol}) and (\ref{estimationofm}), we obtain there is $\eta\in W^s_{\text{loc}}(x, f)$ such that
\begin{eqnarray}\label{stablees}
\begin{aligned}
d_s(f^{-m}x, f^{-m}y)&\geq m(D_\eta f^{-m}|_{E^s_\eta}) \cdot d_s(x,y)\\
&\geq \prod_{k=0}^{p'-1} m(D_{f^{-kN_0}\eta}f^{-N_0}|_{E^s_{f^{-kN_0}\eta}}) \cdot \tilde{\kappa}_0 \cdot d_s(x, y)\\
&\geq \prod_{k=0}^{p'-1} m(D_{f^{-kN_0}x}f^{-N_0}|_{E^s_{f^{-kN_0}x}}) \cdot e^{-m\varepsilon} \cdot \tilde{\kappa}_0 \cdot d_s(x, y)\\
&\geq \prod_{k=0}^{p'-1} |\det(D_{f^{-kN_0}x}f^{-N_0}|_{E^s_{f^{-kN_0}x}})|^{\frac{1}{d_1}} \cdot e^{-2m\varepsilon} \cdot \tilde{\kappa}_0 \cdot d_s(x, y)\\
&\geq |\det (D_xf^{-m}|_{E^s_x})|^{\frac{1}{d_1}} \cdot e^{-2m\varepsilon} \cdot \tilde{\kappa}_0\tilde{\kappa}_1^{-N_0} \cdot d_s(x, y)\\
&= \tilde{\kappa} e^{-S_{-m}\psi^-(x)} \cdot e^{-2m\varepsilon} \cdot d_s(x, y)\\
&\geq \tilde{\kappa} e^{-\log r} \cdot e^{-2m\varepsilon} \cdot d_s(x, y),
\end{aligned}
\end{eqnarray}
where $d_s$ is the metric induced by the Riemannian structure on the stable manifold, $m=p'N_0+q'$ and $0\leq q'<N_0$, $\tilde{\kappa}=\tilde{\kappa}_0\tilde{\kappa}_1^{-N_0}$.
It follows from (\ref{estimationofm}) that $m\leq -\frac{1}{\log\tilde{\kappa}_2}\log r+1$. Therefore combining with (\ref{stablees}) one has
\begin{eqnarray*}
\begin{aligned}
d_s(f^{-m}x, f^{-m}y)&\geq \tilde{\kappa} r^{-1} e^{\frac{2\varepsilon}{\log\tilde{\kappa}_2}\log r - 2\varepsilon} \cdot d_s(x,y)\\
&\geq \kappa' \cdot r^{-1+\frac{2\varepsilon}{\log\tilde{\kappa}_2}} \cdot d_s(x,y).
\end{aligned}
\end{eqnarray*}
Thus
$$d_s(x, y)\leq (\kappa')^{-1} \cdot r^{1-\frac{2}{\log\tilde{\kappa}_2}\varepsilon}\cdot d_s(f^{-m}x, f^{-m}y).$$
Put $\kappa_2=(\kappa')^{-1}\text{diam} \mathcal{P}$ and $\kappa_3=\displaystyle\frac{2}{\log\tilde{\kappa}_2}$. Therefore $|R|\leq \kappa_2 \cdot r^{1-\kappa_3\varepsilon}$. Similarly, $|R|\geq \kappa_1 \cdot r^{1+\kappa_3\varepsilon}$. This completes the proof of the claim.

\end{proof}
By Theorem~1.2 of \cite{Pesin}, with this claim, the cover in the definition of Hausdorff dimension can be taken from $\mathcal{R_{\rho}}$ only.

We proceed to prove  Proposition \ref{key}. By a slight lack of precision (in fact, the proof is valid when we drop the constants for convenience),  we write
\begin{equation}\label{diameterofr}
r^{1+\kappa_3\varepsilon} \leq |R| \leq r^{1-\kappa_3\varepsilon}
\end{equation}
for $R\in R(r)$.
Note that $\theta>\dim_H A$ implies
$$ \theta > \dim_H(A\cap\Gamma_\varepsilon).$$
Hence there exists $\rho_0>0$ such that for every $\rho\in(0, \rho_0)$, $N_1^\pm(\rho)\geq N_1(\rho)\geq N_0$ and
\begin{equation}\label{hausdorffd}
\mathcal{H}_\rho^{\theta}(A\cap\Gamma) \leq \sum_{\ell}|R_\ell|^{\theta} \leq 1,
\end{equation}
where $\{R_\ell\}_\ell$ is some appropriately chosen countable cover of $A\cap\Gamma_\varepsilon$ by rectangles
$$R_\ell\in \mathcal{R}_\rho,\quad R_\ell=\pi(C_\ell)\quad \text{for some}\quad C_\ell\in \mathcal{C}_\rho.$$

For every $R_\ell\in\mathcal{R}_\rho$, there is some $r_\ell\in(0, \rho)$ such that $R_\ell\in R(r_\ell)$. Thus there exists some corresponding finite sequence
$$C=[i_{-n_\ell^-}^\ell \cdots i_{-1}^\ell. i_0^\ell \cdots i_{n_\ell^+-1}^\ell] \in C(r_\ell)$$
such that $R_\ell=\pi(C)$. It follows from (\ref{diameterofr}), (\ref{estimationofn}) and (\ref{2}) that
\begin{eqnarray*}
\begin{aligned}
\log |R_\ell|&\geq (1+\kappa_3\varepsilon)\log r_\ell\\
&>(1+\kappa_3\varepsilon) S_{n_\ell^+}\psi^+(\pi(\underline{i}^\ell))\\
&>-(\chi^u+2\varepsilon)n_\ell^+\cdot (1+\kappa_3\varepsilon)
\end{aligned}
\end{eqnarray*}
for some $\underline{i}^\ell\in[i_{-n_\ell^-}^\ell \cdots i_{-1}^\ell. i_0^\ell \cdots i_{n_\ell^+-1}^\ell]$. Thus combining with (\ref{hausdorffd}) we can estimate
\begin{eqnarray}\label{star9}
\begin{aligned}
\sum_\ell e^{-(\chi^u+2\varepsilon) n_\ell^+ \theta (1+\kappa_3\varepsilon)}&<\sum_\ell |R_\ell|^{\theta}\\
&\leq 1.
\end{aligned}
\end{eqnarray}
By (\ref{2}), (\ref{estimationofpsi-}) and (\ref{estimationofpsi+}), one has
\begin{eqnarray*}
\begin{aligned}
n_\ell^-(\chi^s-2\varepsilon)&<S_{-n_\ell^-}\psi^-(\pi(\underline{i}^\ell))\\
                             &<\log r_\ell\\
                             &\leq S_{n_\ell^+}\psi^+(\pi(\underline{i}^\ell)) + C_0\\
                             &<-n_\ell^+(\chi^u-2\varepsilon) + C_0.
\end{aligned}
\end{eqnarray*}
This implies
$$n_\ell^+\chi^u + n_\ell^-\chi^s < 2\varepsilon (n_\ell^+ + n_\ell^-) + C_0.$$
Therefore
\begin{eqnarray}\label{star10}
\begin{aligned}
-n_\ell^+ \frac{\chi^u}{\chi^s} &= n_\ell^- - n_\ell^+\frac{\chi^u}{\chi^s} - n_\ell^-\frac{\chi^s}{\chi^s}\\
                                &=n_\ell^- - \frac{n_\ell^+\chi^u + n_\ell^-\chi^s}{\chi^s}\\
                                &< n_\ell^- - \frac{2\varepsilon(n_\ell^+ + n_\ell^-) + C_0}{\chi^s}.
\end{aligned}
\end{eqnarray}
Hence combining with (\ref{LEformular}), (\ref{delta}), (\ref{varepsilon}) and (\ref{N0}) we obtain
\begin{eqnarray*}
\begin{aligned}
(1+\kappa_3\varepsilon)\theta n_\ell^+\chi^u &=(1+\kappa_3\varepsilon)\theta n_\ell^+\chi^u\cdot \frac{h_\mu(f)}{\dim_H\mu}\left(\frac{1}{\chi^u}- \frac{1}{\chi^s}\right)\\
                      &=(1+\kappa_3\varepsilon) h_\mu(f)\cdot \frac{\theta}{\dim_H\mu}\left(n_\ell^+- n_\ell^+\frac{\chi^u}{\chi^s}\right)\\
                      &< h_\mu(f) \frac{\theta}{\dim_H\mu}(1+\kappa_3\varepsilon) \left[n_\ell^+ + n_\ell^- -\frac{2\varepsilon(n_\ell^-+n_\ell^+)+C_0}{\chi^s}\right]\\
                      &= h_\mu(f)(n_\ell^+ + n_\ell^-)\frac{\theta}{\dim_H\mu}(1+\kappa_3\varepsilon) \left[1+\frac{2\varepsilon}{|\chi^s|}+\frac{C_0}{|\chi^s|(n_\ell^+ + n_\ell^-)}\right]\\
                      &\leq h_\mu(f)(n_\ell^+ + n_\ell^-)\left(1-\frac\delta2\right)\left(1+\frac\delta4+\frac\delta4\right)\\
                      &= h_\mu(f)(n_\ell^+ + n_\ell^-)\left(1-\frac{\delta^2}{4}\right).
\end{aligned}
\end{eqnarray*}
Thus
\begin{eqnarray}\label{starstar}
\begin{aligned}
&(1+\kappa_3\varepsilon) \theta n_\ell^+ (\chi^u + 2\varepsilon)\\
\leq& h_\mu(f) (n_\ell^+ + n_\ell^-)\left[ 1-\frac{\delta^2}{4} + \frac{2\varepsilon n_\ell^+\theta(1+\kappa_3\varepsilon)}{h_\mu(f)(n_\ell^+ + n_\ell^-)}\right]\\
<& h_\mu(f) (n_\ell^+ + n_\ell^-)\left[ 1-\frac{\delta^2}{4} + \frac{2\varepsilon \theta(1+\kappa_3\varepsilon)}{h_\mu(f)}\right]\\
<& h_\mu(f) (n_\ell^+ + n_\ell^-)\left( 1-\frac{\delta^2}{4} + \frac{\delta^2}{8}\right)\\
=& (n_\ell^+ + n_\ell^-)h_\mu(f) \left( 1- \frac{\delta^2}{8} \right),
\end{aligned}
\end{eqnarray}
the last inequality is by (\ref{varepsilon}). Consider now the family of cylinders of length $n_\ell^+ + n_\ell^-$ given by
\begin{eqnarray*}
\begin{aligned}
\mathcal{U}=\{U_\ell: &U_\ell=\pi^+(\sigma^{-n_\ell^-}(C_\ell)),\\
& \text{ where } C_\ell\in\mathcal{C}_\rho \text{ are taken from the family used in the estimate (\ref{hausdorffd})}  \}.
\end{aligned}
\end{eqnarray*}
Denote $s=h_\mu(f)\left(1-\frac{\delta^2}{8}\right)$. It follows from (\ref{measurableentropy}) together with (\ref{varepsilon}) that
$$h(\sigma, \Sigma_p)=h(\sigma^+, \Sigma_p^+)>s.$$
By (\ref{starstar}) and (\ref{star9}), one has
\begin{eqnarray*}
\begin{aligned}
\sum_{U_\ell\in\mathcal{U}} e^{-s<U_\ell>} &= \sum_{U_\ell\in\mathcal{U}} e^{-s(n_\ell^-+n_\ell^+)}\\
                                           &\leq \sum_{U_\ell\in\mathcal{U}} e^{-(\chi^u +2\varepsilon)\theta n_\ell^+(1+\kappa_3\varepsilon)}\\
                                           &<1.
\end{aligned}
\end{eqnarray*}
For $(\sigma^+, \Sigma_p^+)$, there is some function $H: \mathbb{N}\to \mathbb{R}$ such that
$$\lim_{n\to\infty}H(n)=h(\sigma^+, \Sigma_p^+).$$
Recall that
$$I^{+}(\mathcal{U})=\{\underline{i}\in \Sigma^{+}_{p}: \text{ for every } n, m \in\mathbb{N}\cup\{0\} \text{ with } n<m, \text{ we have } (i_{n}\cdots i_{m})\notin \mathcal{U}\}.$$
Since $\mathcal{U}\subset \bigcup_{n\geq 2N_0} \Sigma_{p,n}^+$, then Proposition \ref{Do2} implies that
\begin{equation}\label{estimationofentropy}
h(\sigma^+, I^+(\mathcal{U}))\geq H(N_0).
\end{equation}
Let
$$I(\mathcal{U})\triangleq \{\underline{i}\in\Sigma_p: \pi^+(\underline{i})\in I^+(\mathcal{U})\}.$$

\begin{Claim}
$\pi\left(I(\mathcal{U})\right)\subset I^+_{f|_{\Gamma_\varepsilon}}(A\bigcap\Gamma_\varepsilon)$.
\end{Claim}

\begin{proof}
For every $x\in \pi\left(I(\mathcal{U})\right)$, there is $\underline{i}\in I(\mathcal{U})$ such that $\pi(\underline{i})=x$. Therefore $\pi^+(\underline{i})\in I^+(\mathcal{U})$. Then we obtain $(\sigma^+)^k\circ \pi^+(\underline{i}) \notin U_\ell$ for each $U_\ell\in\mathcal{U}$ and $k\geq 0$. Since $U_\ell=\pi^+(\sigma^{-n_\ell^-}(C_\ell))$ and $\pi^+\circ \sigma =\sigma^+\circ \pi^+$, then one has
$$(\sigma^+)^k\circ \pi^+(\underline{i}) \notin (\sigma^+)^{-n_\ell^-} \circ \pi^+(C_\ell)\quad\text{for~each~} k\geq 0.$$
Hence $\sigma^{k+n_\ell^-}(\underline{i})\notin C_\ell$ for every $k\geq 0$. This yields that
$$\pi(\sigma^{k+n_\ell^-}(\underline{i})) \notin \pi(C_\ell) =R_\ell \text{ for each } k\geq 0.$$
Combining with $\pi\circ\sigma = f\circ\pi$, we conclude $f^{k+n_\ell^-}(x)\notin R_\ell$ for every $k\geq 0$. Hence $\omega_f(x)\cap R_\ell =\varnothing$ for each $R_\ell$. Since $\{R_\ell\}$ covers $A\cap\Gamma_\varepsilon$, then
$$\omega_f(x)\cap (A\cap\Gamma_\varepsilon) =\varnothing.$$
This shows that $x\in I^+_{f|_{\Gamma_\varepsilon}}(A\bigcap\Gamma_\varepsilon)$.

\end{proof}

The monotonicity of entropy, the above claim and (\ref{estimationofentropy}) imply
$$h(f|_{\Gamma_\varepsilon}, I^{+}_{f|_{\Gamma_\varepsilon}}(A\cap \Gamma_\varepsilon))\geq h(f|_{\Gamma_\varepsilon}, \pi(I(\mathcal{U}))) =h(\sigma, I(\mathcal{U})) \geq h(\sigma^+, I^+(\mathcal{U})) \geq H(N_{0}).$$
Letting $N_{0}$ tend to infinity, we get that ~$h(f|_{\Gamma_\varepsilon}, I^{+}_{f|_{\Gamma_\varepsilon}}(A\cap \Gamma_\varepsilon))=h(f|_{\Gamma_\varepsilon})$.

\end{proof}

\subsection{Proof of Theorem \ref{main2}}
(i)
Recall that $\chi^s(\mu)$ and $\chi^u(\mu)$ are the negative and positive Lyapunov exponents of $\mu$ respectively. Wang and Cao \cite{WC} proved that
$$\dim_H \mu = h_\mu(f) \left(\frac{1}{\chi^u(\mu)} - \frac{1}{\chi^s(\mu)}\right).$$
Since $\dim_H\mu>0$, then $h_\mu(f)>0$.
By Proposition \ref{horseshoe} and Remark \ref{rem3}, for any given  $\varepsilon > 0$,  there exist $\delta(\varepsilon)>0$ and a $(\mu, \varepsilon)$ average conformal horseshoe  $W_{\varepsilon}\subset W$ such that
\begin{align*}
h(f|_{W_{\varepsilon}})\geq h_{\mu}(f)-\varepsilon, \quad \dim_{H}(W_{\varepsilon}\cap W_{\text{loc}}^{\ast}(x, f))\geq \frac{h_{\mu}(f)}{|\chi^{\ast}(\mu)|}-\delta(\varepsilon)\quad  \text{and }
\end{align*}
\begin{align}\label{starstarstarstarstar}
\lim_{\varepsilon\rightarrow 0}\dim_{H} W_{\varepsilon}=\dim_{H}\mu,
\end{align}
where  $\ast = s, u$.
Applying Proposition \ref{key} to  $(W_{\varepsilon}, f)$, we get that
$$h(f|_{W_{\varepsilon}}, I^{+}_{f|_{W_{\varepsilon}}}(A\cap W_{\varepsilon}))=h(f|_{W_{\varepsilon}})\geq h_{\mu}(f)-\varepsilon.$$
Since  $I^{+}_{f|_W}(A)\supset I^{+}_{f|_{W_{\varepsilon}}}(A\cap W_{\varepsilon})$, then
$$h(f|_{W}, I^{+}_{f|_{W}}(A))\geq h(f|_{W_{\varepsilon}}, I^{+}_{f|_{W_{\varepsilon}}}(A\cap W_{\varepsilon}))\geq h_{\mu}(f)-\varepsilon.$$
The arbitrariness  of  $\varepsilon > 0$ implies that
$$h(f|_{W}, I^{+}_{f|_{W}}(A))\geq h_{\mu}(f).$$

Fix any $x\in I^+_{f|_{W_\varepsilon}}(A\cap W_\varepsilon)$. For every $y\in W_\varepsilon\cap W_{\text{loc}}^s(x, f)$, one gets $\omega_f(y)\subset \omega_f(x)$. Therefore $y\in I^+_{f|_{W_\varepsilon}}(A\cap W_\varepsilon)$. This shows that
$$W_\varepsilon\cap W_{\text{loc}}^s(x, f) \subset I^+_{f|_{W_\varepsilon}}(A\cap W_\varepsilon).$$
Hence $I^{+}_{f|_{W_{\varepsilon}}}(A\cap W_{\varepsilon})$ is  $s$-{\it saturated} invariant. Remark $3.2$ and Remark $3.3$ in \cite{WWCZ} told us that
\begin{eqnarray*}
\begin{aligned}
&\ \dim_{H}\left(W_{\varepsilon}\cap W^{u}_{\text{loc}}(x, f)\right)\\
=&\ \sup\left\{\frac{h_\nu(f)}{\chi^u(\nu)}: \nu\in\mathcal{M}_{inv}(f|_{W_\varepsilon})\right\}\\
\leq&\ \frac{h(f|_{W_\varepsilon})}{\chi^u(\mu)-\varepsilon}.
\end{aligned}
\end{eqnarray*}
Combining with Proposition \ref{relation2} and Proposition \ref{key}, one has for any  $x\in I^{+}_{f|_{W_{\varepsilon}}}(A\cap W_{\varepsilon})$,
\begin{eqnarray}\label{starstarstar}
\begin{aligned}
&\ \dim_{H}\left(I^{+}_{f|_{W_{\varepsilon}}}(A\cap W_{\varepsilon})\cap W^{u}_{\text{loc}}(x, f)\right)\\
=&\ \displaystyle\frac{\dim_{H}\left(I^{+}_{f|_{W_{\varepsilon}}}(A\cap W_{\varepsilon})\cap W^{u}_{\text{loc}}(x, f)\right)}{\dim_{H}\left(W_{\varepsilon}\cap W^{u}_{\text{loc}}(x, f)\right)}\cdot \dim_{H}\left(W_{\varepsilon}\cap W^{u}_{\text{loc}}(x, f)\right)\\
\geq&\ \displaystyle\frac{h(f|_{W_{\varepsilon}}, I^{+}_{f|_{W_{\varepsilon}}}(A\cap W_{\varepsilon}))}{h(f|_{W_{\varepsilon}} )}\frac{\chi^{u}(\mu)-\varepsilon}{\chi^{u}(\mu)+\varepsilon}\ \dim_{H}\left(W_{\varepsilon}\cap W^{u}_{\text{loc}}(x, f)\right)\\
=&\ \displaystyle\frac{\chi^{u}(\mu)-\varepsilon}{\chi^{u}(\mu)+\varepsilon}\cdot \dim_{H}\left(W_{\varepsilon}\cap W^{u}_{\text{loc}}(x, f)\right).
\end{aligned}
\end{eqnarray}
It is obvious  that for any  $y\in I^{+}_{f|_{W_{\varepsilon}}}(A\cap W_{\varepsilon})\cap W^{u}_{\text{loc}}(x, f)$,
\begin{equation*}
I^{+}_{f|_{W_{\varepsilon}}}(A\cap W_{\varepsilon})\supset W_{\varepsilon}\cap W^{s}_{\text{loc}}(y, f).
\end{equation*}
Let
$$D_x\triangleq \left\{[y,z]: y \in I^{+}_{f|_{W_{\varepsilon}}}(A\cap W_{\varepsilon})\cap W^{u}_{\text{loc}}(x, f), \ z \in W_{\varepsilon}\cap W^{s}_{\text{loc}}(x, f)\right\},$$
where $[y,z] = W_{\text{loc}}^s(y, f) \cap W^u_{\text{loc}}(z, f)$. Then we obtain
\begin{equation}\label{starstarstarstar}
D_x \subset I^{+}_{f|_{W_{\varepsilon}}}(A\cap W_{\varepsilon}).
\end{equation}

It follows from Theorem $3.3$ in \cite{WWCZ} that for any $r\in (0,1)$, there exists $C_r > 0$ such that both the stable holonomy maps  and the unstable holonomy maps  are $(C_r, r)$ H$\ddot{\text{o}}$lder continuous. Thus under a $(C_r, r)$ H$\ddot{\text{o}}$lder continuous local coordinate transform, we can view  $W_{\varepsilon}$ as the direct product of the local unstable leaf and local stable leaf.  By Lemma \ref{H} and the arbitrariness of $r\in (0,1)$, we have
$$\dim_H D_x = \dim_H \left( ( I^{+}_{f|_{W_{\varepsilon}}}(A\cap W_{\varepsilon})\cap W^{u}_{\text{loc}}(x, f) ) \times \left(W_{\varepsilon}\cap W^{s}_{\text{loc}}(x, f)\right) \right).$$
Now we apply Lemma  \ref{dimensionlemma} to $B_{1}=I^{+}_{f|_{W_{\varepsilon}}}(A\cap W_{\varepsilon})\cap W^{u}_{\text{loc}}(x, f)$, $B_{2}=W_{\varepsilon}\cap W^{s}_{\text{loc}}(x, f)$ and $E=D_x$ that
$$\dim_H D_x \ge \dim_H ( I^{+}_{f|_{W_{\varepsilon}}}(A\cap W_{\varepsilon})\cap W^{u}_{\text{loc}}(x, f))  + \dim_H (W_{\varepsilon}\cap W^{s}_{\text{loc}}(x, f)).$$
Therefore
\begin{align*}
\dim_{H} I^{+}_{f|_W}(A) & \ge  \dim_H I^{+}_{f|_{W_{\varepsilon}}}(A\cap W_{\varepsilon}) \\
&\ge \dim_H D_x \\
&\ge \dim_{H}(W_{\varepsilon}\cap W^{s}_{\text{loc}}(x, f))+\frac{\chi^{u}(\mu)-\varepsilon}{\chi^{u}(\mu)+\varepsilon} \cdot \dim_{H}(W_{\varepsilon}\cap W^{u}_{\text{loc}}(x, f)),
\end{align*}
the second inequality is by (\ref{starstarstarstar}), and the third inequality is by (\ref{starstarstar}).
By (\ref{starstarstarstarstar})
and
$$\dim_{H}(W_{\varepsilon}\cap W^{s}_{\text{loc}}(x, f))+\dim_{H}(W_{\varepsilon}\cap W^{u}_{\text{loc}}(x, f))=\dim_{H}~W_{\varepsilon},$$
letting $\varepsilon \to 0$, one has
$$\dim_{H}~I^{+}_{f|_W}(A)\geq \dim_{H}~\mu.$$
Since $I^+_{f|_W}(A) = E^+_{f|_W}(A) \cup \bigcup_{n\geq0}f^{-n}A$, where $\tilde{A}=\{x\in A:\omega_{f}(x)\cap A=\varnothing\}$, then
$$\dim_{H}~I^{+}_{f|_W}(A)=\max_{n\geq 0}\left\{\dim_{H}~E^{+}_{f|_W}(A), \dim_{H}f^{-n}(\tilde{A})\right\}.$$
It follows from the property of the Hausdorff dimension that
$$\dim_{H} f^{-n}(\tilde{A})=\dim_{H} \tilde{A}\leq \dim_{H} A < \dim_{H}\mu.$$
Thus
$$\dim_{H} E^{+}_{f|_{W}}(A)\geq \dim_{H} \mu.$$

(ii)It follows from the assumption that $h(f|_{W},A)<h_{\mu}(f|_{W})$ that $h_\mu(f)>0$. Then the results follows from the same argument as that in the proof of (i).


\subsection{Proof of Corollary \ref{ddd}}
By the definition of the dynamical dimension and the fact $f$ is average conformal on $W$, there exists a sequence of conformal hyperbolic ergodic measures ~$\{\mu_{n}\}_{n}$ supported on ~$W$ satisfying that $h_{\mu_{n}}(f|_{W})>0$ such that
\begin{displaymath}
\lim_{n\rightarrow \infty }\dim_{H}~\mu_{n}=DD(f|_{W}).
\end{displaymath}
Thus for  $n>0$ large enough, we have
$$\dim_{H}A<\dim_{H}~\mu_{n}\leq DD(f|_{W}).$$
By Theorem \ref{main2},
$$\dim_{H}E^{+}_{f|_{W}}(A)\geq \dim_{H}~\mu_{n}.$$
Letting  $n$ tend to infinity, we have  $\dim_{H}E^{+}_{f|_{W}}(A)\geq DD(f|_{W})$.\\

{\bf Acknowledgements.} Authors are grateful for Professor Yongluo Cao's suggestions which led to a great improvement of the manuscript.

\end{document}